\documentclass[10pt]{article}

\usepackage{cmap}

\usepackage[T1]{fontenc}
\usepackage{lmodern}
\usepackage[a4paper,margin=2cm]{geometry}
\usepackage{amsmath,amssymb,amsthm,mathrsfs}
\usepackage{microtype}
\usepackage{enumitem}
\usepackage[hidelinks]{hyperref}
\usepackage[nameinlink,noabbrev]{cleveref}

\allowdisplaybreaks
\numberwithin{equation}{section}

\newtheorem{theorem}{Theorem}[section]
\newtheorem{proposition}[theorem]{Proposition}
\newtheorem{lemma}[theorem]{Lemma}
\newtheorem{corollary}[theorem]{Corollary}
\theoremstyle{definition}
\newtheorem{definition}[theorem]{Definition}
\theoremstyle{remark}
\newtheorem{remark}[theorem]{Remark}

\Crefname{theorem}{Theorem}{Theorems}
\Crefname{proposition}{Proposition}{Propositions}
\Crefname{lemma}{Lemma}{Lemmas}
\Crefname{corollary}{Corollary}{Corollaries}
\Crefname{definition}{Definition}{Definitions}
\Crefname{remark}{Remark}{Remarks}

\newcommand{\R}{\mathbb{R}}
\newcommand{\Id}{\mathrm{Id}}
\newcommand{\Lip}{\operatorname{Lip}}
\newcommand{\supp}{\operatorname{supp}}
\newcommand{\diam}{\operatorname{diam}}
\newcommand{\conv}{\operatorname{conv}}
\newcommand{\tr}{\operatorname{tr}}
\newcommand{\Cov}{\operatorname{Cov}}
\newcommand{\ran}{\operatorname{ran}}
\newcommand{\Ker}{\operatorname{Ker}}
\newcommand{\bigO}{\mathcal{O}}
\newcommand{\essran}{\operatorname*{ess\,ran}}
\newcommand{\e}{\mathrm{e}}
\newcommand{\dd}{\,\mathrm{d}}
\newcommand{\cA}{\mathcal{A}}
\newcommand{\cH}{\mathcal{H}}
\newcommand{\cK}{\mathcal{K}}
\newcommand{\cL}{\mathcal{L}}
\newcommand{\cS}{\mathcal{S}}
\newcommand{\cC}{\mathfrak{C}}
\newcommand{\cM}{\mathfrak{M}}
\newcommand{\ip}[2]{\left\langle #1,#2\right\rangle}
\newcommand{\norm}[1]{\left\lvert #1\right\rvert}
\newcommand{\opnorm}[1]{\left\lVert #1\right\rVert_{\mathrm{op}}}

\title{Caffarelli Estimates under Lipschitz Perturbations}
\author{Maja Gw\'o\'zd\'z\\
ETH Z\"urich\\
\href{mailto:mgwozdz@ethz.ch}{\texttt{mgwozdz@ethz.ch}}}
\date{}

\begin{document}

\frenchspacing

\maketitle

\begin{abstract}
We study dimension-free differential estimates for Brenier maps under first-order perturbations of the two marginals. Caffarelli's contraction theorem yields such estimates when the source and target potentials satisfy a pointwise Hessian comparison. We show that the same conclusion remains true under arbitrary globally Lipschitz perturbations of both marginals, without additional assumptions on convexity, semiconvexity, or smallness on the perturbations. In particular, this answers the conjecture due to Fathi, Mikulincer, and Shenfeld. 

More precisely, let \(d\ge1\), \(L\ge0\), and let \(B:\R^d\to\R\) be globally \(L\)-Lipschitz. We prove that the Brenier map from the standard Gaussian measure \(\gamma_d\) to the probability measure proportional to \(\e^{-B}\gamma_d\) has a globally Lipschitz representative whose bound depends only on \(L\). The Brenier potential also belongs to \(C^{1,1}(\R^d)\) and satisfies dimension-free two-sided Hessian bounds. Let \(\cC(0,L)\) denote the upper-bound constant, then
\[
\log\cC(0,L)=4L^2+\log L+\bigO(1)
\qquad(L\to\infty),
\]
and a one-dimensional example with Brenier map \(T_L\) satisfies \(\log\Lip(T_L)\ge L^2/2\). This shows that the quadratic order is optimal.

We deduce this Gaussian estimate from an anisotropic two-marginal result. Let
\(V,W:\R^d\to\R\) be potentials, and let \(Q,P\) be positive-definite matrices. Using the distributional curvature bounds \(D^2V\preceq Q\) and \(D^2W\succeq P\), we establish matrix Hessian bounds for Brenier maps between arbitrary globally Lipschitz perturbations of the reference marginals. Most importantly, the constants keep the directional geometry of the essential gradient ranges. For affine perturbations, our estimates recover the sharp noncommuting Caffarelli tensor. We further obtain pointwise displacement bounds for Gaussian perturbations and spectral Hessian estimates for compact Gaussian mixtures.
\end{abstract}

\begingroup\small
\par
\endgroup\medskip

\section{Introduction}\label{sec:introduction}

Let us fix an integer \(d\ge1\). We write \(\norm{\cdot}\) for the Euclidean norm, \(\Id\) for the identity matrix and the identity map, and \(\preceq\) for the Loewner order. The standard Gaussian probability measure on \(\R^d\) is
\[
\dd\gamma_d(x)
:=(2\pi)^{-d/2}\e^{-\norm{x}^2/2}\dd x.
\]
For maps between Euclidean spaces, we use \(\Lip(\,\cdot\,)\) for the Euclidean
Lipschitz constant.

Caffarelli's contraction theorem provides a global Lipschitz estimate for quadratic optimal transport. In scalar form, if \(D^2V\preceq\Lambda\Id\) and \(D^2W\succeq\kappa\Id\), then the quadratic-cost Brenier map from \(\e^{-V}\dd x\) to \(\e^{-W}\dd x\), after normalisation and under the finite-moment assumptions, is \(\sqrt{\Lambda/\kappa}\)-Lipschitz. Caffarelli's original normalised theorem transports a Gaussian measure to a more log-concave target \cite[Theorem 11]{Caffarelli} (see also \cite{CaffarelliErratum} and \cite[Theorem 5.5]{GS} for the matrix Hessian comparison). We ask whether dimension-free Brenier-map estimates still hold under perturbations that are controlled only at first order.

\subsection{Gaussian conjecture}
\label{subsec:gaussian-conjecture}

Fathi, Mikulincer, and Shenfeld \cite[Conjecture 1]{FMS} conjectured that a dimension-free Lipschitz estimate should hold for the quadratic-cost Brenier map under Gaussian
log-Lipschitz perturbations. More precisely, for every \(L\ge0\), they ask for a constant \(C(L)<\infty\), independent of \(d\), with the following property. Let \(B:\R^d\to\R\) satisfy \(\Lip(B)\le L\), set \(Z_B:=\int_{\R^d}\e^{-B}\dd\gamma_d\), and let \(T\) be the Brenier map from \(\gamma_d\) to \(Z_B^{-1}\e^{-B}\gamma_d\). The question is whether
\[
\Lip(T)\le C(L)
\]
holds. Notice that Caffarelli's theorem does not apply directly, because a globally
Lipschitz function \(B\) need not be semiconvex.

For \(\rho>0\) and \(\ell_a,\ell_b\ge0\), we define
\begin{equation}\label{eq:def-Crho}
\cC_\rho(\ell_a,\ell_b)
:=\inf_{0<\beta<1}
\frac1\beta
\exp\!\left(
\frac{4(\ell_b+\beta\ell_a)^2}{\rho(1-\beta^2)}
\right),
\end{equation}
and write
\begin{equation}\label{eq:def-CAB}
\cC(\ell_a,\ell_b):=\cC_1(\ell_a,\ell_b),
\qquad
\cC_\rho(0,0):=1.
\end{equation}

\begin{theorem}[Gaussian log-Lipschitz Brenier estimate]
\label{thm:gaussian-source}
Let \(L\ge0\), let \(B:\R^d\to\R\) be globally \(L\)-Lipschitz, and set
\[
Z_B:=\int_{\R^d}\e^{-B}\dd\gamma_d,
\qquad
\dd\nu_B=Z_B^{-1}\e^{-B}\dd\gamma_d.
\]
Let \(T=\nabla\Phi\) be the Brenier map from \(\gamma_d\) to \(\nu_B\). It follows that \(\Phi\) may be chosen in \(C^{1,1}(\R^d)\), that \(T\) has a globally defined Lipschitz representative, and that
\begin{equation}\label{eq:gaussian-source-Hessian-intro}
\cC(L,0)^{-1}\Id\preceq D^2\Phi\preceq\cC(0,L)\Id
\end{equation}
in the sense of matrix-valued distributions and almost everywhere. In particular, \(\Lip(T)\le\cC(0,L)\), and
\[
\log\cC(0,L)
=
4L^2+\log L+\bigO(1)
\qquad(L\to\infty).
\]
\end{theorem}

The quadratic order of \(\log\cC(0,L)\) is already necessary in dimension one. For \(L>0\), let \(B^{(L)}(t)=-L|t|\), and denote the associated monotone transport by \(T_L\). It follows that
\[
\Lip(T_L)\ge \e^{L^2/2}.
\]
We prove this estimate and discuss the normalisation in \Cref{rem:neeman-obstruction}. This shows optimality of the quadratic order, although the leading constant in the upper logarithmic asymptotic is not sharp.

\subsection{Anisotropic theorem and applications}
\label{subsec:anisotropic-theorem-applications}

\Cref{thm:gaussian-source} follows from the two-marginal estimate in \Cref{thm:spectral-caffarelli}. In this theorem, we study arbitrary globally Lipschitz perturbations of both marginals under one-sided distributional matrix-curvature bounds \(D^2V\preceq Q\) and \(D^2W\succeq P\). Its constants keep the directional geometry of the essential gradient ranges, and affine perturbations recover the sharp noncommuting Caffarelli tensor \(G\) determined by the equation \(GPG=Q\). Scalar choices of the auxiliary metric give the ellipsoidal and isotropic results we discuss in \Cref{cor:matrix-caffarelli,cor:perturbative-caffarelli}. The application of the same estimate to the reverse transport yields the two-sided Gaussian bound in \Cref{cor:gaussian-bilip}.

We also prove two Gaussian applications. First, \Cref{thm:gaussian-displacement}
identifies the pointwise displacement of the Brenier map between two Gaussian
log-Lipschitz perturbations through the gradient ranges of the perturbations. Moreover, for the Brenier map from \(\gamma_d\) to a compact Gaussian mixture \(\gamma_d*\eta\), \Cref{thm:mixtures} gives a dimension-free Hessian estimate controlled by the geometry of \(\supp\eta\). In \Cref{sec:geometry,sec:anisotropic-proof}, we prove the anisotropic theorem and derive its one-sided corollaries. \Cref{sec:two-sided-gaussian} presents the forward--reverse argument, the Gaussian constants, and the displacement estimate. \Cref{sec:mixtures} is about compact Gaussian mixtures, and \Cref{app:exact-schur-envelopes} proves the exact spectral Schur envelope for the coercivity argument.

\section{Related works}
\label{sec:related-work}

Fathi, Mikulincer, and Shenfeld \cite[Theorem 1]{FMS} proved that the Kim--Milman Langevin transport map satisfies a dimension-free Lipschitz estimate in the
log-Lipschitz perturbation settings. They also proposed a conjecture for the quadratic-cost Brenier map \cite[Conjecture 1]{FMS}. They further described the Gaussian case as a good starting point and identified the noncompact setting, possible unboundedness of solutions, and the need for concrete dimension-free estimates as the main difficulties \cite[Section 7.2]{FMS}.

For comparison, let us assume that \(\Lip(a),\Lip(b)\le L\) in the log-Lipschitz
Gaussian setting. The two-sided estimate of Gozlan and Sylvestre \cite[Theorem 5.15]{GS} gives
\[
\norm{x-y}-8L
\le\norm{T_{a,b}(x)-T_{a,b}(y)}
\le\norm{x-y}+8L
\qquad(x,y\in\R^d).
\]
Observe that the additive approximate-isometry bound does not control the differential quotient as \(y\to x\), so it does not provide the Lipschitz estimate for the Brenier map.

There exist multiple results in the literature that provide dimension-free Lipschitz or contraction estimates for heat-flow, Langevin, and other diffusion transports
under curvature, first-order, or H\"older hypotheses. See
\cite{KimMilman,Neeman,BrigatiPedrotti,ConfortiEichinger,LopezRivera,StephanovitchFlow} for a selection of those studies, and consult \cite{ChewiEichingerPooladian} for quantitative stability estimates for the Kim--Milman flow map. We stress that these estimates analyse transports that are not defined by quadratic-cost optimality. Moreover, the Kim--Milman heat-flow map can differ from the Brenier map \cite{Tanana}, so to the best of our knowledge, \Cref{thm:gaussian-source} gives the first dimension-free differential estimate for the canonical quadratic-cost transport in the log-Lipschitz Gaussian regime. The anisotropic comparison also has a very rich history. For example, Valdimarsson\cite[Theorem 1.2]{Valdimarsson} established an estimate for the Gaussian-convolution source data under the partial commutation hypothesis, and Chewi and Pooladian \cite[Theorems 13--14 and Remark 15]{ChewiPooladian} extended this result to general source and target measures that satisfy the commuting matrix curvature bounds. Finally, Gozlan and Sylvestre \cite[Theorem 5.5]{GS} obtained the fully noncommuting tensor estimate. 

Other global estimates for Brenier maps are known under different structural
assumptions. For instance, Colombo, Figalli, and Jhaveri \cite{ColomboFigalliJhaveri} obtained Lipschitz changes of variables for compactly supported perturbations under additional structural assumptions. The trace analogue of Caffarelli's contraction
theorem for log-subharmonic source measures and strongly log-concave targets
is due to De Philippis and Shenfeld \cite{DePhilippisShenfeld}, while recent
results by Ammari and Figalli \cite{AmmariFigalli} and Bidoia \cite{Bidoia} establish global Brenier-map estimates in special structured regimes. Ammari and Figalli \cite{AmmariFigalli} study a dimension-free interpolation between polynomial-type
and log-concave densities, and Bidoia \cite{Bidoia} obtains bounds for good convex, monotone, positively homogeneous functionals under the associated bounds. Further global estimates for Brenier maps under assumptions other than our regime are given in \cite{KolesnikovHessian,ColomboFathi,CarlierFigalliSantambrogio}. For growth
estimates obtained from concentration inequalities, consult \cite{FathiGrowth}.

For functional inequalities concerning Gaussian convolutions of compactly
supported measures, see \cite{Zimmermann,BardetGozlanMalrieuZitt}. In dimension one, Zimmermann proves a logarithmic Sobolev inequality by estimating the derivative of the monotone rearrangement from the Gaussian, which is precisely the Brenier map
\cite[Section 2]{Zimmermann}. Higher-dimensional transport estimates for
Gaussian mixtures are known for the finite-dimensional heat-flow map
\cite[Theorem 2]{MikulincerShenfeldHeat} and for the Brownian transport map
from Wiener measure to the mixture \cite[Theorem 1.3]{MikulincerShenfeldBrownian}. Other relevant results under Gaussian convolution are given in \cite{KlartagPutterman},  parabolic log-Hessian estimates appear in \cite{ChaintronConfortiEichinger}, while the anisotropic \(W_\infty\) transport estimates are proved in \cite{KhudiakovaMaasPedrotti}.

Finally, let us recall some finite-difference results we use in our proofs. Valdimarsson
applies a translated finite-difference corrector of the same type \cite[Section 2, eqs. (2.1)--(2.6)]{Valdimarsson}, and Kolesnikov proposes a related multidimensional argument \cite[Theorem 2.5, eqs. (7)--(8)]{KolesnikovHolder}. We combine these ideas with the joint maximum principle and the compatible pseudoinverse Schur-complement idea we have already introduced in our previous work \cite[Sections 2, 4, and 8.1, especially Lemma 8.1]{Gwozdz}. The most important extension is that in
\cite{Gwozdz}, we relied on compact target range to obtain the coercivity needed for the joint maximum principle. In this work, we consider the full-support context that replaces that range bound by the one-sided affine modulus from \Cref{lem:global-modulus}.

\section{Main anisotropic estimate}
\label{sec:geometry}

Let us first establish the necessary notation and preliminaries. We write \(\mathrm{Sym}_d\) for the real symmetric \(d\times d\) matrices, \(\mathrm{Sym}_d^{++}\) for the positive-definite cone, and \(\mathbb S^{d-1}\) for the Euclidean unit sphere. For a nonempty compact convex set \(K\subset\R^d\), we define the support function \(h_K:\R^d\to\R\) by
\[
h_K(u):=\sup_{z\in K}\ip{z}{u}.
\]

For \(M\in\mathrm{Sym}_d^{++}\), we write
\begin{equation}\label{eq:matrix-norm}
\norm{x}_M:=\sqrt{\ip{Mx}{x}},
\qquad
\Lip_M(f):=
\sup_{x\ne y}\frac{|f(x)-f(y)|}{\norm{x-y}_M}
\end{equation}
for \(x\in\R^d\) and scalar-valued \(f:\R^d\to\R\). For scalar-valued \(f\), this reduces to
\[
\Lip(f)=\Lip_{\Id}(f).
\]
For \(K,L\subset\R^d\) and \(v\in\R^d\), we use the notation
\[
\begin{aligned}
K-L&:=\{p-q:p\in K,\ q\in L\},
&\qquad
K+v&:=\{p+v:p\in K\},\\
K-v&:=\{p-v:p\in K\},
&\qquad
-K&:=\{-p:p\in K\}.
\end{aligned}
\]

\begin{definition}[Gradient-width body]\label{def:gradient-width}
Let \(f:\R^d\to\R\) be globally Lipschitz. We set
\begin{equation}\label{eq:gradient-body}
\cK_f
:=\overline{\conv}\bigl(\essran\nabla f\bigr).
\end{equation}
Here, \(\nabla f\) is the almost-everywhere defined weak gradient, and we take
\(\essran\) with respect to Lebesgue measure on \(\R^d\). We further define the associated gradient-width gauge \(\omega_f:\R^d\to[0,\infty)\) by
\begin{equation}\label{eq:gradient-width}
\omega_f(h)
:=\sup_{p,q\in\cK_f}\ip{p-q}{h}
=h_{\cK_f-\cK_f}(h).
\end{equation}
\end{definition}

Let \(Q,P\in\mathrm{Sym}_d^{++}\) and set
\begin{equation}\label{eq:G-main}
G:=P^{-1/2}(P^{1/2}QP^{1/2})^{1/2}P^{-1/2}.
\end{equation}
Squaring \(P^{1/2}GP^{1/2}\) gives
\begin{equation}\label{eq:GPG-main}
GPG=Q.
\end{equation}
Note that uniqueness among positive-definite solutions follows directly from uniqueness of the positive square root. The balanced curvature matrix is
\begin{equation}\label{eq:R-main}
\mathsf R
:=G^{-1/2}QG^{-1/2}
=G^{1/2}PG^{1/2}.
\end{equation}
For \(H\in\mathrm{Sym}_d^{++}\), we define
\begin{equation}\label{eq:D-H-main}
\mathsf D_H
:=H^{1/2}\mathsf RH^{1/2}
-H^{-1/2}\mathsf RH^{-1/2}.
\end{equation}
For \(r\in\R\), we set
\[
r_+:=\max\{r,0\}.
\]
Let \(a,b:\R^d\to\R\) be globally Lipschitz. For \(u\in\mathbb S^{d-1}\), we define
\begin{align}
n_{Q,P}^{a,b,\mathrm{red}}(H,u)
&:={}
\omega_b(G^{1/2}H^{1/2}u)
+\omega_a(G^{-1/2}H^{-1/2}u),
\label{eq:n-red-main}\\
\cA_{Q,P}^{\mathrm{red}}(H,a,b)
&:={}
\int_0^\infty
\sup_{u\in\mathbb S^{d-1}}
\left(
n_{Q,P}^{a,b,\mathrm{red}}(H,u)
-\frac t2\ip{\mathsf D_Hu}{u}
\right)_+\dd t.
\label{eq:spectral-action-red-main}
\end{align}
Observe that
\begin{equation}\label{eq:gauge-admissibility-main}
H^{1/2}\mathsf D_HH^{1/2}=H\mathsf RH-\mathsf R,
\end{equation}
so \(\mathsf D_H\succ0\) if and only if \(H\mathsf RH\succ\mathsf R\). Under
this condition, the reduced spectral action is finite. The admissible class is also nonempty. Indeed, for every \(c>1\), the matrix \(H=c\Id\) is admissible because \(H\mathsf RH=c^2\mathsf R\succ\mathsf R\).

Let \(F:\R^d\to\R\) be finite and continuous, and let \(A\in\mathrm{Sym}_d\). Let \(D_v^2F\) denote the second distributional derivative of \(F\) in the direction \(v\in\R^d\). The relation \(D^2F\preceq A\) (in the sense of distributions) means that
\[
\big\langle D_v^2F,\varphi\big\rangle
\le
\ip{Av}{v}\int_{\R^d}\varphi(x)\,\dd x
\]
for every \(v\in\R^d\) and every nonnegative \(\varphi\in C_c^\infty(\R^d)\). The reverse inequality defines \(D^2F\succeq A\).

\begin{theorem}[Spectral perturbative Caffarelli theorem]
\label{thm:spectral-caffarelli}
Let \(Q,P\in\mathrm{Sym}_d^{++}\), and let \(V,W:\R^d\to\R\) be finite continuous functions such that, in the sense of distributions,
\begin{equation}\label{eq:matrix-curvature-main}
D^2V\preceq Q,
\qquad
D^2W\succeq P.
\end{equation}
Let \(a,b:\R^d\to\R\) be globally Lipschitz, and set
\[
Z_{\mu,a}:=\int_{\R^d}\e^{-V(x)-a(x)}\dd x,
\qquad
Z_{\nu,b}:=\int_{\R^d}\e^{-W(y)-b(y)}\dd y.
\]
We assume that \(0<Z_{\mu,a},Z_{\nu,b}<\infty\) and that the probability measures
\begin{equation}\label{eq:matrix-marginals-main}
\dd\mu_a(x)=Z_{\mu,a}^{-1}\e^{-V(x)-a(x)}\dd x,
\qquad
\dd\nu_b(y)=Z_{\nu,b}^{-1}\e^{-W(y)-b(y)}\dd y
\end{equation}
have finite second moments. Let \(T=\nabla\Phi\) be the quadratic-cost Brenier map from \(\mu_a\) to \(\nu_b\). Its existence and uniqueness follow from the Brenier--McCann theorem \cite{Brenier,McCann}.

For every \(H\succ0\) such that \(H\mathsf RH\succ\mathsf R\), it follows that
\begin{equation}\label{eq:spectral-upper-main}
0\preceq D^2\Phi
\preceq
\exp\!\bigl(\cA_{Q,P}^{\mathrm{red}}(H,a,b)\bigr)
G^{1/2}HG^{1/2}
\end{equation}
in the sense of matrix-valued distributions and almost everywhere. In particular, the Brenier potential may be chosen in \(C^{1,1}(\R^d)\), and \(T=\nabla\Phi\) has a globally defined Lipschitz representative. For this representative, we set
\begin{equation}\label{eq:M-H-main}
M_H
:=\exp\!\bigl(\cA_{Q,P}^{\mathrm{red}}(H,a,b)\bigr)
G^{1/2}HG^{1/2},
\end{equation}
so that
\begin{equation}\label{eq:spectral-metric-map-main}
\norm{T(x)-T(y)}_{M_H^{-1}}
\le\norm{x-y}_{M_H}
\qquad(x,y\in\R^d).
\end{equation}
\end{theorem}

Notice that the matrices in the family \eqref{eq:spectral-upper-main} need not be comparable in the Loewner order. Nevertheless, for every fixed \(v\in\R^d\), the scalar distribution \(D_v^2\Phi\) satisfies
\begin{equation}\label{eq:directional-envelope-main}
D_v^2\Phi
\le
\inf_{\substack{H\succ0\\H\mathsf RH\succ\mathsf R}}
\e^{\cA_{Q,P}^{\mathrm{red}}(H,a,b)}
\ip{G^{1/2}HG^{1/2}v}{v}.
\end{equation}
For a fixed direction \(v\), the associated pointwise inequality holds almost everywhere once we take a minimising sequence of admissible metrics.

\begin{corollary}[Affine perturbations and sharpness]
\label{cor:affine-sharpness}
For \(m\in\R^d\) and \(\Sigma\in\mathrm{Sym}_d^{++}\), let \(\mathcal N(m,\Sigma)\) be the Gaussian probability measure on \(\R^d\) with mean \(m\) and covariance matrix \(\Sigma\). Under the hypotheses of \Cref{thm:spectral-caffarelli}, assume that \(a\) and \(b\) are affine. It follows that
\begin{equation}\label{eq:unperturbed-matrix-sharp}
D^2\Phi\preceq G.
\end{equation}
For \(a=b=0\), equality is attained by the Gaussian pair \(\mathcal N(0,Q^{-1})\) and \(\mathcal N(0,P^{-1})\).
\end{corollary}

\begin{corollary}[Ellipsoidal Lipschitz perturbations]
\label{cor:matrix-caffarelli}
With the hypotheses of \Cref{thm:spectral-caffarelli}, we set
\begin{equation}\label{eq:R-rho-main}
\rho:=\lambda_{\min}(\mathsf R),
\qquad
\ell_a:=\Lip_G(a),
\qquad
\ell_b:=\Lip_{G^{-1}}(b).
\end{equation}
It follows that
\begin{equation}\label{eq:matrix-upper-main}
0\preceq D^2\Phi\preceq\cC_\rho(\ell_a,\ell_b)G.
\end{equation}
In other terms,
\begin{equation}\label{eq:matrix-metric-map-main}
\norm{T(x)-T(y)}_{G^{-1}}
\le\cC_\rho(\ell_a,\ell_b)\norm{x-y}_G
\qquad(x,y\in\R^d).
\end{equation}
\end{corollary}

\begin{remark}[Euclidean Lipschitz perturbations]\label{rem:euclidean-matrix}
We assume that \(L_a,L_b\ge0\), \(\Lip(a)\le L_a\), and \(\Lip(b)\le L_b\), and obtain
\[
\Lip_G(a)\le\frac{L_a}{\sqrt{\lambda_{\min}(G)}},
\qquad
\Lip_{G^{-1}}(b)\le L_b\sqrt{\lambda_{\max}(G)}.
\]
The estimate in \Cref{cor:matrix-caffarelli} depends on the perturbations via these two ellipsoidal Lipschitz constants. By contrast, \Cref{thm:spectral-caffarelli} keeps their directional gradient widths.
\end{remark}

\begin{corollary}[Scalar perturbative Caffarelli estimate]
\label{cor:perturbative-caffarelli}
Let \(D^2V\preceq\Lambda\Id\) and \(D^2W\succeq\kappa\Id\), where \(\Lambda,\kappa>0\). Under the hypotheses of \Cref{thm:spectral-caffarelli}, we obtain
\begin{equation}\label{eq:perturbative-upper-main}
0\preceq D^2\Phi
\preceq
\sqrt{\frac{\Lambda}{\kappa}}\,
\cC\!\left(
\frac{\Lip(a)}{\sqrt\Lambda},
\frac{\Lip(b)}{\sqrt\kappa}
\right)\Id.
\end{equation}
\end{corollary}

For \(x_0\in\R^d\) and \(r\ge0\), let \(\overline B_r(x_0):=\{x\in\R^d:\norm{x-x_0}\le r\}\) denote the closed Euclidean ball.

\subsection{Gradient-width geometry and balancing}
\label{subsec:gradient-width-geometry-balancing}

\begin{lemma}[First-order remainders and gradient widths]
\label{lem:width-Bregman}
Let \(f:\R^d\to\R\) be globally Lipschitz.
\begin{enumerate}[label=(\roman*)]
\item The gauge \(\omega_f\) is finite, even, convex, and positively homogeneous, and \(\omega_{f+\ell}=\omega_f\) for every affine function \(\ell\).
\item If \(f\in C^1(\R^d)\), then, for all \(x,h\in\R^d\),
\begin{equation}\label{eq:width-Bregman-bounds}
-\omega_f(h)
\le
f(x+h)-f(x)-\ip{\nabla f(x)}{h}
\le\omega_f(h).
\end{equation}
\item If \(L\ge0\) and \(\Lip(f)\le L\), then, for every \(h\in\R^d\),
\begin{equation}\label{eq:width-Lip-bound}
\omega_f(h)\le2L\norm h.
\end{equation}
\item Let \((\zeta_\varepsilon)_{\varepsilon>0}\) be a standard mollifier family and set \(f_\varepsilon:=f*\zeta_\varepsilon\). We then obtain
\begin{equation}\label{eq:width-mollification}
\cK_{f_\varepsilon}\subseteq\cK_f,
\qquad
\omega_{f_\varepsilon}\le\omega_f.
\end{equation}
\end{enumerate}
\end{lemma}

\begin{proof}
The identity \(\omega_f=h_{\cK_f-\cK_f}\) gives the first statement. Notice that translating \(\cK_f\) by the gradient of an affine function does not change its difference body. Assume now that \(f\in C^1(\R^d)\). We obtain
\[
f(x+h)-f(x)-\ip{\nabla f(x)}{h}
=\int_0^1\ip{\nabla f(x+th)-\nabla f(x)}{h}\dd t.
\]
Since \(\nabla f\) is continuous, every value \(\nabla f(z)\) belongs to \(\essran\nabla f\). Indeed, the inverse image of each neighbourhood of \(\nabla f(z)\) contains a neighbourhood of \(z\) and has positive Lebesgue measure. Both gradients in the integrand belong to \(\cK_f\), which gives \eqref{eq:width-Bregman-bounds}. If \(f\) is \(L\)-Lipschitz, then \(\cK_f\subseteq\overline B_L(0)\), which proves
\eqref{eq:width-Lip-bound}. Finally, recall that the weak gradient of a Lipschitz function satisfies
\[
\nabla f_\varepsilon(x)
=\int\nabla f(x-z)\zeta_\varepsilon(z)\dd z
\]
for every \(x\). For almost every \(z\), the integrand belongs to \(\cK_f\), which implies that \(\nabla f_\varepsilon(x)\in\cK_f\) and \(\cK_{f_\varepsilon}\subseteq\cK_f\). The support-function identity now gives \(\omega_{f_\varepsilon}\le\omega_f\).
\end{proof}

We set
\begin{equation}\label{eq:matrix-change-variables}
\widetilde x=G^{1/2}x,
\qquad
\widetilde y=G^{-1/2}y,
\end{equation}
and define the transformed reference and perturbation potentials by
\begin{align*}
\widetilde V(\widetilde x)
&:=
V(G^{-1/2}\widetilde x),
&
\widetilde a(\widetilde x)
&:=
a(G^{-1/2}\widetilde x),
\\
\widetilde W(\widetilde y)
&:=
W(G^{1/2}\widetilde y),
&
\widetilde b(\widetilde y)
&:=
b(G^{1/2}\widetilde y).
\end{align*}
It is possible to absorb the constant Jacobian factors into the normalising constants. The transformed potential is
\begin{equation}\label{eq:matrix-scaled-potential}
\widetilde\Phi(\widetilde x)
:=\Phi(G^{-1/2}\widetilde x).
\end{equation}
We then have
\begin{equation}\label{eq:matrix-scaled-map}
\nabla\widetilde\Phi(\widetilde x)
=G^{-1/2}T(G^{-1/2}\widetilde x),
\qquad
D^2\widetilde\Phi(\widetilde x)
=G^{-1/2}D^2\Phi(G^{-1/2}\widetilde x)G^{-1/2}.
\end{equation}
The map in \eqref{eq:matrix-scaled-map} pushes the transformed source measure
forward to the transformed target measure. As \(\widetilde\Phi\) is convex, Brenier--McCann uniqueness \cite{Brenier,McCann} identifies \(\nabla\widetilde\Phi\) with the quadratic-cost Brenier map for the transformed pair. The transformed reference potentials also satisfy
\begin{equation}\label{eq:balanced-curvature}
D^2\widetilde V\preceq\mathsf R,
\qquad
D^2\widetilde W\succeq\mathsf R.
\end{equation}
The associated width gauges are
\begin{equation}\label{eq:balanced-widths}
\omega_{\widetilde a}(h)=\omega_a(G^{-1/2}h),
\qquad
\omega_{\widetilde b}(h)=\omega_b(G^{1/2}h).
\end{equation}
It now suffices to prove \Cref{thm:spectral-caffarelli} in the balanced coordinates, where the source and target share the curvature matrix \(\mathsf R\). Up to and including \eqref{eq:balanced-final-spectral}, we use the symbols \(V,W,a,b,T,\Phi\) and the corresponding marginals for these objects with balanced coordinates.

\subsection{Regularity and large-scale coercivity}
\label{subsec:regularity-large-scale-coercivity}

\begin{proposition}[Regularity on the full space]
\label{prop:smooth-diffeo-general}
Let \(V,W,a,b\in C^\infty(\R^d)\) be finite, and assume that the measures with densities proportional to \(\e^{-V-a}\) and \(\e^{-W-b}\) have finite second moments. It follows that the Brenier potential may be chosen finite on all of \(\R^d\), that
\[
T=\nabla\Phi:\R^d\longrightarrow\R^d
\]
is a \(C^\infty\) diffeomorphism, and that \(D^2\Phi(x)\succ0\) for every \(x\).
\end{proposition}

\begin{proof}
In dimension one, the conclusion follows from monotone rearrangement and the
inverse-function theorem. Let us now assume that \(d\ge2\). Notice that both densities
are smooth and strictly positive. On every compact subset of \(\R^d\), the densities and their reciprocals are bounded, while both supports equal \(\R^d\). We may now apply \cite[Corollary 1]{CorderoFigalli}, which gives the global homeomorphism
\[
T=\nabla\Phi:\R^d\longrightarrow\R^d.
\]
For every integer \(k\ge0\) and every \(\alpha\in(0,1)\), the densities belong to
\(C_{\mathrm{loc}}^{k,\alpha}(\R^d)\). The local regularity part of the same corollary then gives
\[
T\in C_{\mathrm{loc}}^{k+1,\alpha}(\R^d,\R^d).
\]
If we vary the differentiability order, we have \(T\in C^\infty(\R^d,\R^d)\). Since \(T=\nabla\Phi\) and \(\Phi\) is convex, we obtain \(D^2\Phi\succeq0\). The Monge--Amp\`ere identity
\[
\det D^2\Phi(x)
=
c\,\exp\!\bigl(
-V(x)-a(x)+W(T(x))+b(T(x))
\bigr),
\qquad c>0,
\]
holds almost everywhere. By continuity of both sides, it extends to every
point. Its right-hand side is strictly positive, so \(D^2\Phi(x)\succ0\) for every \(x\in\R^d\). This implies that \(T\) is a smooth local diffeomorphism. We now combine this result with the global homeomorphism above, and conclude that \(T\) is a smooth diffeomorphism of \(\R^d\).
\end{proof}

Let us now set
\begin{equation}\label{eq:R-eigenvalues}
\rho:=\lambda_{\min}(\mathsf R),
\qquad
\overline\rho:=\lambda_{\max}(\mathsf R),
\end{equation}
and write
\begin{equation}\label{eq:auxiliary-Euclidean-Lipschitz}
L_a:=\Lip(a),
\qquad
L_b:=\Lip(b).
\end{equation}

\begin{lemma}[Large-scale Brenier modulus]\label{lem:global-modulus}
Assume that \(V,W,a,b\) are smooth, \(a,b\) are globally Lipschitz, that the measures with densities proportional to \(\e^{-V-a}\) and \(\e^{-W-b}\) are probability measures with finite second moments, and that
\begin{equation}\label{eq:balanced-standing}
D^2V\preceq\mathsf R,
\qquad
D^2W\succeq\mathsf R.
\end{equation}
It follows that there exist finite constants \(c_0,c_1\), which depend only on \(\mathsf R,L_a,L_b\), such that
\begin{equation}\label{eq:global-affine-modulus}
\norm{T(x)-T(y)}\le c_0+c_1\norm{x-y}
\qquad(x,y\in\R^d).
\end{equation}
For the translated corrector
\begin{equation}\label{eq:G-def-prelim}
\mathfrak B_m(x):=\Phi(x+m)-\Phi(x)-\ip{m}{T(x)},
\end{equation}
we also obtain
\begin{equation}\label{eq:G-global-growth}
0\le\mathfrak B_m(x)
\le c_0\norm m+\frac{c_1}{2}\norm m^2.
\end{equation}
\end{lemma}

\begin{proof}
We write
\[
Z_{\mu,a}
:=
\int_{\R^d}\e^{-V(x)-a(x)}\dd x,
\qquad
Z_{\nu,b}
:=
\int_{\R^d}\e^{-W(y)-b(y)}\dd y,
\]
and set
\[
U_\mu:=V+a+\log Z_{\mu,a},
\qquad
U_\nu:=W+b+\log Z_{\nu,b}.
\]
Let us fix \(x_0,x_1\in\R^d\). For \(t\in[0,1]\), we set
\[
x_t:=(1-t)x_0+tx_1,
\qquad
r:=\norm{x_1-x_0}.
\]
The curvature and Lipschitz bounds imply
\begin{align*}
U_\mu(x_t)
&\ge
(1-t)U_\mu(x_0)+tU_\mu(x_1)
-t(1-t)\left(
\frac{\overline\rho}{2}r^2+2L_ar
\right),\\
U_\nu(x_t)
&\le
(1-t)U_\nu(x_0)+tU_\nu(x_1)
-t(1-t)\left(
\frac{\rho}{2}r^2-2L_br
\right).
\end{align*}
We define
\[
\sigma_\mu(r)
:=
\frac{\overline\rho}{2}r^2+2L_ar,
\qquad
\rho_\nu(r)
:=
\frac{\rho}{2}r^2-2L_br.
\]
We follow the terminology of \cite[Section 2]{GS}, and say that the normalised total potential \(U_\mu\) is \(\sigma_\mu\)-smooth, while \(U_\nu\) is \(\rho_\nu\)-convex. The
measures \(\e^{-U_\mu}\dd x\) and \(\e^{-U_\nu}\dd y\) are probabilities, so we may apply \cite[Corollary 4.3]{GS} directly. The calculation in the proof of \cite[Theorem 5.16, equation (31)]{GS} then gives, for \(x,y\in\R^d\) and \(r:=\norm{x-y}\),
\[
\norm{T(x)-T(y)}
\le
\frac{4L_b}{\rho}
+
4\sqrt{
\frac{L_b^2}{\rho^2}
+
\frac{L_a}{\rho}r
}
+
\sqrt{\frac{\overline\rho}{\rho}}\,r.
\]
For \(u,v\ge0\), we use
\[
\sqrt{u+v}\le\sqrt u+\sqrt v
\qquad\text{and}\qquad
\sqrt r\le\frac{1+r}{2},
\]
and infer that
\[
\norm{T(x)-T(y)}
\le c_0+c_1r,
\]
for suitable finite \(c_0,c_1\), which depend only on \(\mathsf R,L_a,L_b\). Convexity gives \(\mathfrak B_m\ge0\). On the other hand,
\[
\mathfrak B_m(x)
=
\int_0^1
\ip{T(x+tm)-T(x)}{m}\dd t.
\]
We apply \eqref{eq:global-affine-modulus} with \(y=x+tm\) and integrate in \(t\), which gives
\[
0\le\mathfrak B_m(x)
\le c_0\norm m+\frac{c_1}{2}\norm m^2.
\]
The lemma follows immediately.
\end{proof}

For a finite convex function \(\Psi:\R^d\to\R\), we use
\[
\Psi^*(y):=\sup_{x\in\R^d}
\bigl\{\ip{x}{y}-\Psi(x)\bigr\}
\]
for its Fenchel conjugate, and use \(\partial\Psi\) for its convex subdifferential.

\begin{lemma}[Onto gradient implies supercoercivity]\label{lem:supercoercive}
Let \(\Psi:\R^d\to\R\) be finite, differentiable, and convex. If \(\nabla\Psi(\R^d)=\R^d\), then
\begin{equation}\label{eq:supercoercive}
\frac{\Psi(x)}{\norm x}\longrightarrow+\infty
\qquad(\norm x\to\infty).
\end{equation}
\end{lemma}

\begin{proof}
Let us fix \(y\in\R^d\) and choose \(x_y\in\R^d\) such that \(\nabla\Psi(x_y)=y\). By the Fenchel equality,
\[
\Psi^*(y)=\ip{x_y}{y}-\Psi(x_y)<\infty.
\]
We deduce that \(\Psi^*\) is finite, convex, and continuous. For \(R>0\), set
\[
M_R:=\sup_{\norm y\le R}\Psi^*(y)<\infty.
\]
If \(x\ne0\), Fenchel duality with \(y=Rx/\norm x\) gives \(\Psi(x)\ge R\norm x-M_R\). We obtain
\[
\liminf_{\norm x\to\infty}\frac{\Psi(x)}{\norm x}\ge R.
\]
Since \(R>0\) is arbitrary, \eqref{eq:supercoercive} holds.
\end{proof}

\begin{lemma}[Closure of Hessian bounds]\label{lem:C11}
Let \(\Psi:\R^d\to\R\) be finite and convex.
\begin{enumerate}[label=(\roman*),ref=(\roman*)]
\item\label{lem:C11-upper} Let us assume that, for some \(M\in\mathrm{Sym}_d^{++}\),
\begin{equation}\label{eq:matrix-second-difference-upper}
\Psi(x+h)+\Psi(x-h)-2\Psi(x)\le\ip{Mh}{h}
\end{equation}
for all \(x,h\). It follows that \(\Psi\in C^{1,1}(\R^d)\) and \(D^2\Psi\preceq M\) in distributions and almost everywhere.
\item\label{lem:C11-lower} If, in addition, the reverse inequality holds with some \(N\in\mathrm{Sym}_d^{++}\), then \(D^2\Psi\succeq N\) and
\[
\ip{\nabla\Psi(x)-\nabla\Psi(y)}{x-y}
\ge\ip{N(x-y)}{x-y}
\qquad(x,y\in\R^d).
\]
\end{enumerate}
\end{lemma}

\begin{proof}
We set \(q_M(x):=\ip{Mx}{x}/2\) and \(F:=q_M-\Psi\). Recall that the inequality
\eqref{eq:matrix-second-difference-upper} says that \(F\) is midpoint convex.
Since \(F\) is continuous, it is convex, so we have
\[
0\preceq D^2\Psi\preceq M\preceq\opnorm{M}\Id
\]
in the sense of matrix-valued distributions. Notice that the convexity of \(\Psi\) implies that every directional second derivative \(D_v^2\Psi\) is a nonnegative Radon measure. We combine the upper bound
\[
D_v^2\Psi\le\ip{Mv}{v}\,\dd x
\]
with polarisation to obtain
\[
D^2\Psi=A(x)\,\dd x
\]
for some \(A\in L^\infty(\R^d,\mathrm{Sym}_d)\) such that \(0\preceq A(x)\preceq M\) almost everywhere. We infer that \(\Psi\in W_{\mathrm{loc}}^{2,\infty}(\R^d)\) and that its first weak derivatives admit globally Lipschitz representatives. Observe that convexity identifies these representatives with the classical gradient of \(\Psi\). We conclude that \(\Psi\in C^{1,1}(\R^d)\), and \(0\preceq D^2\Psi\preceq M\) both almost
everywhere and in the sense of distributions. This establishes \ref{lem:C11-upper}.

Let us now assume that the reverse second-difference inequality holds with
\(N\), and set \(q_N(x):=\ip{Nx}{x}/2\). The midpoint-convex function \(\Psi-q_N\) is continuous and thus convex. It follows that \(D^2\Psi\succeq N\) distributionally and almost everywhere. Finally, the monotonicity of \(\nabla(\Psi-q_N)\) gives
\[
\ip{\nabla\Psi(x)-\nabla\Psi(y)}{x-y}
\ge\ip{N(x-y)}{x-y}.
\]
This proves \ref{lem:C11-lower}.
\end{proof}

\begin{lemma}[Metric estimate from a Hessian bound]
\label{lem:hessian-bound-metric}
Let \(\Psi\in C^{1,1}(\R^d)\) be convex, and let \(M\in\mathrm{Sym}_d^{++}\). If \(0\preceq D^2\Psi\preceq M\) almost everywhere, then
\begin{equation}\label{eq:hessian-bound-metric}
\norm{\nabla\Psi(x)-\nabla\Psi(y)}_{M^{-1}}
\le\norm{x-y}_M
\qquad(x,y\in\R^d).
\end{equation}
\end{lemma}

\begin{proof}
We first set \(v:=x-y\). Since \(\nabla\Psi\) is Lipschitz, it is absolutely continuous on line segments, and
\[
\nabla\Psi(x)-\nabla\Psi(y)
=
\int_0^1D^2\Psi(y+tv)v\dd t.
\]
If \(0\preceq A\preceq M\), then \(M^{-1/2}AM^{-1/2}\) lies between \(0\) and
\(\Id\). We have \(AM^{-1}A\preceq M\) and
\[
\norm{Av}_{M^{-1}}\le\norm v_M.
\]
We integrate this estimate with \(A=D^2\Psi(y+tv)\) and obtain \eqref{eq:hessian-bound-metric}.
\end{proof}

\Cref{lem:uniform-Lipschitz-stability} is the full-support analogue of our previous result from \cite[Lemma 7.3]{Gwozdz}. Here, we use the uniform Lipschitz and base-point bounds to replace the compact-range hypothesis.

For a measurable map \(S\) and a Borel probability measure \(\sigma\), we
write \(S_\#\sigma\) for the pushforward of \(\sigma\) under \(S\), and use \(W_2\)
for the quadratic Wasserstein distance and \(\sigma_n\rightharpoonup\sigma\) for weak convergence. For a topological space \(E\), let \(C_b(E)\) denote the bounded continuous functions on \(E\). If \(F\in C_b(E)\), we set
\[
\norm{F}_\infty:=\sup_{z\in E}|F(z)|.
\]

\begin{lemma}[Stability under a uniform Lipschitz bound]
\label{lem:uniform-Lipschitz-stability}
Let \(\mu_n,\nu_n,\mu,\nu\) be probability measures on \(\R^d\) with finite second moments such that
\[
\mu_n\longrightarrow\mu,
\qquad
\nu_n\longrightarrow\nu
\qquad\text{in }W_2.
\]
We assume that \(\mu\) has a strictly positive density on \(\R^d\). For every \(n\),
let \(T_n=\nabla\Phi_n\) be a globally continuous Brenier map from \(\mu_n\) to
\(\nu_n\), where \(\Phi_n(0)=0\). Suppose further that
\[
\sup_n\Lip(T_n)<\infty,
\qquad
\sup_n\norm{T_n(0)}<\infty.
\]
We conclude that \((T_n)\) converges locally uniformly to the globally continuous Brenier representative \(T=\nabla\Phi\) from \(\mu\) to \(\nu\). If we normalise by \(\Phi(0)=0\), then
\[
\Phi_n\longrightarrow\Phi
\qquad\text{locally uniformly on }\R^d.
\]
\end{lemma}

\begin{proof}
Notice that the assumptions imply that \((T_n)\) is locally equibounded and
equi-Lipschitz. By Arzel\`a--Ascoli, every subsequence admits a further subsequence, which we do not relabel, such that
\[
T_n\longrightarrow T_\infty
\qquad\text{locally uniformly}.
\]
Since \(\Phi_n(0)=0\), the segment identity gives
\[
\Phi_n(x)
=
\int_0^1\ip{T_n(tx)}{x}\dd t.
\]
If we pass to the limit in this identity, we obtain local uniform convergence of \(\Phi_n\) to the finite function
\[
\Phi_\infty(x)
:=
\int_0^1\ip{T_\infty(tx)}{x}\dd t.
\]
The functions \(\Phi_n\) are convex, so \(\Phi_\infty\) is convex. We now pass to the limit in
\[
\Phi_n(y)-\Phi_n(x)
=
\int_0^1
\ip{T_n(x+t(y-x))}{y-x}\dd t
\]
This identity shows that \(\Phi_\infty\in C^1(\R^d)\) and \(T_\infty=\nabla\Phi_\infty\).

Let us set
\[
\pi_n:=(\Id,T_n)_\#\mu_n.
\]
For every \(F\in C_b(\R^d\times\R^d)\), the definition gives
\[
\int F\dd\pi_n
=
\int F(x,T_n(x))\dd\mu_n(x).
\]
Since the map \(x\mapsto F(x,T_\infty(x))\) is bounded and continuous,
\(\mu_n\rightharpoonup\mu\) gives
\[
\int F(x,T_\infty(x))\dd\mu_n(x)
\longrightarrow
\int F(x,T_\infty(x))\dd\mu(x).
\]
For every \(R>0\),
\[
\begin{aligned}
&\left|
\int_{\R^d}
\bigl[F(x,T_n(x))-F(x,T_\infty(x))\bigr]
\dd\mu_n(x)
\right|\\
&\qquad\le
\sup_{x\in\overline B_R(0)}
\left|F(x,T_n(x))-F(x,T_\infty(x))\right|
+2\norm{F}_\infty\,
\mu_n\!\left(\overline B_R(0)^{\,c}\right).
\end{aligned}
\]
The first term tends to zero. Indeed, \(T_n\to T_\infty\) uniformly on \(\overline B_R(0)\), the sets \(T_n(\overline B_R(0))\) lie in a common compact set, and \(F\) is uniformly continuous on the resulting compact subset of \(\R^d\times\R^d\). Uniform tightness of \((\mu_n)\) makes the second term uniformly small once \(R\) is sufficiently large. We conclude that
\[
\int F\dd\pi_n
\longrightarrow
\int F(x,T_\infty(x))\dd\mu(x),
\]
or, in other terms,
\[
\pi_n\rightharpoonup(\Id,T_\infty)_\#\mu.
\]

We also have
\[
\left|
W_2(\mu_n,\nu_n)-W_2(\mu,\nu)
\right|
\le
W_2(\mu_n,\mu)+W_2(\nu_n,\nu)
\longrightarrow0.
\]
We observe that the optimal quadratic costs converge. The stability of optimal
plans under \(W_2\) convergence \cite[Theorem 5.20]{Villani} shows that the limiting coupling is optimal between \(\mu\) and \(\nu\). By Brenier--McCann uniqueness \cite{Brenier,McCann}, every locally uniform subsequential limit is identified \(\mu\)-almost everywhere with the Brenier map from \(\mu\) to \(\nu\). This implies that any two such limits agree \(\mu\)-almost everywhere, and thus everywhere by continuity and strict positivity of the density of \(\mu\). We conclude that the locally uniform limit is unique, and it is a globally continuous Brenier representative (denoted by \(T\)). Since every subsequence admits a locally uniformly convergent further subsequence, the full sequence converges locally uniformly to \(T\). The segment identity then gives the local uniform convergence of the normalised potentials.
\end{proof}

\section{Proof of the anisotropic estimate}
\label{sec:anisotropic-proof}

In the proofs below, we work with the following assumptions. We have \(V,W,a,b\in C^\infty(\R^d)\), and the associated marginals are probability measures with finite second moments. The functions \(a,b\) are globally Lipschitz, and \eqref{eq:balanced-standing} holds. It follows from \Cref{prop:smooth-diffeo-general} that the Brenier map \(T=\nabla\Phi\) is a \(C^\infty\) diffeomorphism and that \(D^2\Phi\succ0\) on \(\R^d\).

\subsection{Translated Monge--Amp\`ere identities}
\label{sec:translated-identities}

We set
\begin{equation}\label{eq:total-potentials}
\mathcal V:=V+a,
\qquad
\mathcal W:=W+b,
\end{equation}
and write
\begin{equation}\label{eq:Hessian-A-general}
A(x):=D^2\Phi(x)\succ0.
\end{equation}
The Monge--Amp\`ere equation becomes
\begin{equation}\label{eq:MA-general}
\log\det A(x)
=c-\mathcal V(x)+\mathcal W(T(x)),
\qquad c\in\R.
\end{equation}
For \(f\in C^2(\R^d)\), we define the drifted linearised operator \cite[Sections 2 and 5]{Caffarelli} \cite[Section 2, equations (2.4)--(2.6)]{KolesnikovHessian} by
\begin{equation}\label{eq:L-general}
(\cL f)(x)
:=\tr\!\left(A(x)^{-1}D^2f(x)\right)
-\ip{\nabla\mathcal W(T(x))}{\nabla f(x)}.
\end{equation}
This is the differential linearisation of \eqref{eq:MA-general}. For comparison, Valdimarsson applies the finite-difference linearisation \cite[Section 2, equation (2.6)]{Valdimarsson}. If \(R\) is positive definite, we set
\begin{equation}\label{eq:H-general}
\cH(R):=\tr R-d-\log\det R.
\end{equation}
Following the convention of \cite[formula (1.4)]{Bregman}, \(\cH(R)\) is the directed
Bregman divergence \(D_{-\log\det}(R,\Id)\).

Let us recall the translated corrector \(\mathfrak B_m\) from \eqref{eq:G-def-prelim}, and define
\begin{equation}\label{eq:translation-notation-general}
C_m(x):=A(x+m),
\qquad
\delta_m(x):=T(x+m)-T(x),
\qquad
R_m(x):=A(x)^{-1/2}C_m(x)A(x)^{-1/2}.
\end{equation}

For \(f\in C^1(\R^d)\) and \(z,h\in\R^d\), we define
\begin{equation}\label{eq:Bregman-remainder-notation}
D_f(z+h,z):=f(z+h)-f(z)-\ip{\nabla f(z)}{h}.
\end{equation}

In \Cref{lem:translated-corrector-general}, we extend our previous results
\cite[Lemmas 4.3 and 8.3]{Gwozdz} to two full-support marginals.

\begin{lemma}[Translated identity with gradient widths]
\label{lem:translated-corrector-general}
For every \(x,m\in\R^d\), we have
\begin{equation}\label{eq:LG-exact-general}
\cL\mathfrak B_m
=\cH(R_m)
-D_{\mathcal V}(x+m,x)
+D_{\mathcal W}(T(x)+\delta_m(x),T(x)).
\end{equation}
We obtain
\begin{equation}\label{eq:LG-width-lower}
\cL\mathfrak B_m
\ge\cH(R_m)
-\frac12\ip{\mathsf Rm}{m}-\omega_a(m)
+\frac12\ip{\mathsf R\delta_m}{\delta_m}
-\omega_b(\delta_m).
\end{equation}
Differentiation in the translation variable gives
\begin{align}
\nabla_m\mathfrak B_m(x)&=\delta_m(x),
\label{eq:Gm-first-general}\\
D^2_{mm}\mathfrak B_m(x)&=C_m(x),
\label{eq:Gm-mm-general}\\
D^2_{xm}\mathfrak B_m(x)&=C_m(x)-A(x).
\label{eq:Gm-xm-general}
\end{align}
\end{lemma}

\begin{proof}
We take the difference of \eqref{eq:MA-general} at \(x+m\) and \(x\) to obtain
\[
\log\det R_m
=-\bigl(\mathcal V(x+m)-\mathcal V(x)\bigr)
+\mathcal W(T(x)+\delta_m(x))-\mathcal W(T(x)).
\]
We further have
\[
\cL\bigl(\Phi(\,\cdot+m)-\Phi\bigr)
=\tr R_m-d-\ip{\nabla\mathcal W(T(x))}{\delta_m}.
\]
If we differentiate \eqref{eq:MA-general} in the constant direction \(m\), we get
\[
\cL\ip{m}{T}=-\ip{\nabla\mathcal V(x)}{m}.
\]
It suffices to combine these three identities to give \eqref{eq:LG-exact-general}. Since
\(D^2V\preceq\mathsf R\) and \(D^2W\succeq\mathsf R\), the curvature bounds give
\[
D_V(x+m,x)\le\frac12\ip{\mathsf Rm}{m},
\qquad
D_W(T(x)+\delta_m(x),T(x))
\ge\frac12\ip{\mathsf R\delta_m}{\delta_m}.
\]
We now apply \Cref{lem:width-Bregman} to \(a\) and \(b\). By direct differentiation of \eqref{eq:G-def-prelim}, we arrive at the derivative identities.
\end{proof}

We translate the target so that the unique minimiser of \(W\) is the origin. Recall the notation \(W,b,T,\Phi\) for the translated objects. The bound \(D^2W\succeq\mathsf R\succeq\rho\Id\) gives
\[
\ip{\nabla W(y)}{y}\ge\rho\norm y^2.
\]
We obtain
\begin{equation}\label{eq:LPhi-general-bound}
\cL\Phi
=d-\ip{\nabla(W+b)(T(x))}{T(x)}
\le d-\rho\norm{T(x)}^2+L_b\norm{T(x)}
\le d+\frac{L_b^2}{4\rho}.
\end{equation}

\subsection{Schur coercivity and spectral penalties}
\label{sec:schur}
\label{sec:spectral-penalty}

For a positive-semidefinite matrix \(D\), we use \(D^\dagger\) for its Moore--Penrose pseudoinverse. \Cref{lem:schur} is Albert's generalised Schur-complement criterion for
semidefinite block matrices \cite[Theorem 1(i)]{Albert} (\textit{cf}. \cite[Lemma 8.1]{Gwozdz}).

\begin{lemma}[Generalised Schur complement]\label{lem:schur}
Let \(n,k\) be positive integers, let \(H\in\R^{n\times n}\) and \(D\in\R^{k\times k}\) be symmetric, let \(B\in\R^{n\times k}\), and assume that \(D\succeq0\). If
\begin{equation}\label{eq:block-negative-general}
\begin{pmatrix}
H&B\\ B^{\mathsf T}&-D
\end{pmatrix}\preceq0,
\end{equation}
then
\begin{equation}\label{eq:kernel-compat-general}
Bz=0\qquad(z\in\Ker D)
\end{equation}
and
\begin{equation}\label{eq:schur-conclusion-general}
H+BD^\dagger B^{\mathsf T}\preceq0.
\end{equation}
\end{lemma}

Let \(\Xi_+:(0,\infty)\to[0,\infty)\) be defined by
\[
\Xi_+(y):=
\begin{cases}
0,&0<y\le1,\\[1mm]
y+2\log y-y^{-1},&y>1.
\end{cases}
\]
For \(M\in\mathrm{Sym}_d^{++}\), we use \(\Xi_+(M)\) for the functional-calculus extension of \(\Xi_+\).

\begin{lemma}[Compatible pseudoinverse congruence]
\label{lem:compatible-pseudoinverse-congruence}
Let \(D,E\in\mathrm{Sym}_d\) satisfy \(D\succeq0\) and \(\Ker D\subseteq\Ker E\), and let \(P\in\mathrm{Sym}_d^{++}\). We set
\[
\widetilde D:=P^{-1/2}DP^{-1/2},
\qquad
\widetilde E:=P^{-1/2}EP^{-1/2}.
\]
By direct computation, we obtain
\begin{equation}\label{eq:pseudoinverse-congruence-compatible}
\widetilde E\widetilde D^\dagger\widetilde E
=P^{-1/2}ED^\dagger EP^{-1/2}.
\end{equation}
\end{lemma}

\begin{proof}
Since \(D\) and \(E\) are symmetric, the kernel condition implies \(\ran E\subseteq\ran D\). In particular, \(\ran\widetilde E\subseteq\ran\widetilde D\). We fix \(x\in\R^d\) and set \(z:=\widetilde D^\dagger\widetilde E x\). Because \(\widetilde E x\in\ran\widetilde D\), we have \(\widetilde D z=\widetilde E x\), and so
\[
DP^{-1/2}z=EP^{-1/2}x.
\]
It follows that \(P^{-1/2}z=D^\dagger EP^{-1/2}x+k\) for some \(k\in\Ker D\). Since \(Ek=0\), we apply \(\widetilde E\) and obtain \eqref{eq:pseudoinverse-congruence-compatible}.
\end{proof}

In \cite[Lemma 4.4]{Gwozdz}, we compute the scalar log-determinant--Schur envelope \(\Xi_+\). The anisotropic estimate from \Cref{lem:matrix-coercivity}, which extends our result in \cite[Lemma 8.2]{Gwozdz}, follows from the exact spectral envelope in \Cref{prop:spectral-schur-envelope}.

\begin{lemma}[Log-determinant--Schur coercivity]
\label{lem:matrix-coercivity}
Let \(A,C,P\in\mathrm{Sym}_d^{++}\) satisfy \(C\preceq P\) and
\begin{equation}\label{eq:compatibility-matrix-general}
(C-A)z=0\qquad(z\in\Ker(P-C)).
\end{equation}
We set
\begin{equation}\label{eq:S-def-general}
\cS(A,C,P)
:=\tr\!\left(
A^{-1}(C-A)(P-C)^\dagger(C-A)
\right).
\end{equation}
It follows that, for every nonzero \(v\in\R^d\),
\begin{equation}\label{eq:matrix-coercivity-general}
\cH(A^{-1/2}CA^{-1/2})+\cS(A,C,P)
\ge
\Xi_+\!\left(\frac{\ip{Av}{v}}{\ip{Pv}{v}}\right).
\end{equation}
\end{lemma}

\begin{proof}
Let us set
\[
\widetilde A:=P^{-1/2}AP^{-1/2},
\qquad
\widetilde C:=P^{-1/2}CP^{-1/2},
\qquad
\widetilde v:=P^{1/2}v.
\]
We have \(\widetilde C\preceq\Id\). If \(z\in\Ker(\Id-\widetilde C)\), then \(P^{-1/2}z\in\Ker(P-C)\). This implies that \((\widetilde C-\widetilde A)z=0\).

To compare the Schur terms, we set
\[
D:=P-C,
\qquad
E:=C-A,
\qquad
\widetilde D:=P^{-1/2}DP^{-1/2},
\qquad
\widetilde E:=P^{-1/2}EP^{-1/2}.
\]
The compatibility condition gives \(\Ker D\subseteq\Ker E\). By
\Cref{lem:compatible-pseudoinverse-congruence}, we obtain
\eqref{eq:pseudoinverse-congruence-compatible}. Cyclicity of the trace then gives
\[
\tr\!\left(
\widetilde A^{-1}\widetilde E
\widetilde D^\dagger\widetilde E
\right)
=\cS(A,C,P).
\]
The identity \(\widetilde A^{-1}\widetilde C =P^{1/2}A^{-1}CP^{-1/2}\) shows that
\(\widetilde A^{-1/2}\widetilde C\widetilde A^{-1/2}\) and \(A^{-1/2}CA^{-1/2}\) have the same eigenvalues, and their Rayleigh quotients satisfy
\[
\frac{\ip{\widetilde A\widetilde v}{\widetilde v}}
{\norm{\widetilde v}^2}
=
\frac{\ip{Av}{v}}{\ip{Pv}{v}}.
\]
It now suffices to prove the normalised estimate. Let us first assume that
\(\widetilde C\prec\Id\). By \Cref{prop:spectral-schur-envelope}, which we apply with
\(P=\Id\),
\begin{align*}
&\cH(\widetilde A^{-1/2}\widetilde C\widetilde A^{-1/2})
+\tr\!\left[
\widetilde A^{-1}(\widetilde C-\widetilde A)
(\Id-\widetilde C)^{-1}
(\widetilde C-\widetilde A)
\right]\\
&\qquad\ge \tr\Xi_+(\widetilde A).
\end{align*}
The nonnegativity and monotonicity of \(\Xi_+\) give
\[
\tr\Xi_+(\widetilde A)
\ge
\Xi_+\!\left(\lambda_{\max}(\widetilde A)\right)
\ge
\Xi_+\!\left(
\frac{\ip{\widetilde A\widetilde v}{\widetilde v}}
{\norm{\widetilde v}^2}
\right).
\]

Let us now consider the general case \(\widetilde C\preceq\Id\). We denote by \(K\) the orthogonal projection onto \(\Ker(\Id-\widetilde C)\) and, for \(0<\varepsilon<1\), set
\[
\widetilde C_\varepsilon:=\widetilde C-\varepsilon K.
\]
It follows that
\[
0\prec\widetilde C_\varepsilon\prec\Id,
\]
and the compatibility condition and symmetry imply
\[
(\widetilde C-\widetilde A)K
=
K(\widetilde C-\widetilde A)
=
0.
\]
Here, \(K\) is the spectral projection of \(\Id-\widetilde C\) associated with the eigenvalue zero. The previous identity gives
\[
(\Id-\widetilde C_\varepsilon)^{-1}
=
(\Id-\widetilde C)^\dagger+\varepsilon^{-1}K.
\]
By direct expansion, we have
\begin{align*}
&(\widetilde C_\varepsilon-\widetilde A)
(\Id-\widetilde C_\varepsilon)^{-1}
(\widetilde C_\varepsilon-\widetilde A)\\
&\qquad=
(\widetilde C-\widetilde A)
(\Id-\widetilde C)^\dagger
(\widetilde C-\widetilde A)
+\varepsilon K.
\end{align*}
We apply this the strict case to \(\widetilde C_\varepsilon\) and let
\(\varepsilon\downarrow0\) to get the required limit. Using
\eqref{eq:pseudoinverse-congruence-compatible} and the above Rayleigh
quotient identity, we conclude the proof of \eqref{eq:matrix-coercivity-general}.
\end{proof}

\begin{lemma}[Joint maximum calculus with a quadratic baseline]
\label{lem:quadratic-baseline-calculus}
Let \(\Gamma\in\mathrm{Sym}_d\) and \(\Theta\in C^2(\R^d)\), and define
\[
h_\Gamma(x):=\Phi(x)-\frac12\ip{\Gamma x}{x},
\qquad
S_\Gamma(x):=T(x)-\Gamma x,
\]
\[
F^\Gamma_m(x)
:=h_\Gamma(x+m)-h_\Gamma(x)-\ip{m}{S_\Gamma(x)}.
\]
We assume that
\[
J_\varepsilon(x,m):=F^\Gamma_m(x)-\Theta(m)-\varepsilon\Phi(x)
\]
has a local maximum at \((x,m)\), and set
\[
A:=D^2\Phi(x),\qquad
C:=D^2\Phi(x+m),\qquad
\delta:=T(x+m)-T(x),\qquad
\Pi:=\Gamma+D^2\Theta(m).
\]
If \(\Pi\succ0\), then
\begin{align}
\delta&=\Gamma m+\nabla\Theta(m),
\label{eq:baseline-delta-stationarity}\\
Am&=\delta-\varepsilon T(x).
\label{eq:baseline-x-stationarity}
\end{align}
The same calculation gives
\begin{equation}\label{eq:baseline-compatibility}
C\preceq\Pi,\qquad
(C-A)z=0\quad(z\in\Ker(\Pi-C)),
\end{equation}
and
\begin{equation}\label{eq:baseline-L-Schur}
\cL F^\Gamma_m-\varepsilon\cL\Phi+\cS(A,C,\Pi)\le0.
\end{equation}
\end{lemma}

\begin{proof}
The identity
\[
F^\Gamma_m(x)=\mathfrak B_m(x)-\frac12\ip{\Gamma m}{m}
\]
gives
\[
\nabla_mF^\Gamma_m=\delta-\Gamma m,\qquad
\nabla_xF^\Gamma_m=\delta-Am,
\]
and
\[
D^2_{mm}F^\Gamma_m=C-\Gamma,\qquad
D^2_{xm}F^\Gamma_m=C-A.
\]
The stationarity of \(J_\varepsilon\) gives \eqref{eq:baseline-delta-stationarity} and
\eqref{eq:baseline-x-stationarity}. At the maximum, the Hessian condition is
\[
\begin{pmatrix}
D^2_{xx}F^\Gamma_m-\varepsilon A&C-A\\
C-A&C-\Pi
\end{pmatrix}\preceq0.
\]
It follows from the lower-right block that \(C\preceq\Pi\). We now apply \Cref{lem:schur} and obtain the compatibility in \eqref{eq:baseline-compatibility}, along with
\[
D^2_{xx}F^\Gamma_m-\varepsilon A
+(C-A)(\Pi-C)^\dagger(C-A)\preceq0.
\]
Finally, we take the \(A^{-1}\)-trace and use \(\nabla_xF^\Gamma_m=\varepsilon T(x)\), which yields \eqref{eq:baseline-L-Schur}.
\end{proof}

With the Schur coercivity estimate ready, we now construct the spectral penalty needed for the maximum-principle argument. To this end, we fix \(S\in\mathrm{Sym}_d^{++}\) and \(0<\beta<1\) such that
\begin{equation}\label{eq:balanced-admissibility}
\mathsf D_{\beta,S}
:=S^{1/2}\mathsf RS^{1/2}
-\beta^2S^{-1/2}\mathsf RS^{-1/2}
\succ0.
\end{equation}
For \(u\in\mathbb S^{d-1}\), define
\begin{align}
\alpha_S(u)
&:=\ip{S^{1/2}\mathsf RS^{1/2}u}{u},
\label{eq:alpha-S}\\
\gamma_S(u)
&:=\ip{S^{-1/2}\mathsf RS^{-1/2}u}{u},
\label{eq:gamma-S}\\
d_{\beta,S}(u)
&:=\alpha_S(u)-\beta^2\gamma_S(u),
\label{eq:d-beta-S-balanced}\\
n_{\beta,S}(u)
&:=\omega_b(S^{1/2}u)
+\beta\omega_a(S^{-1/2}u).
\label{eq:n-beta-S-balanced}
\end{align}
Let
\begin{equation}\label{eq:balanced-spectral-action}
\cA_{\mathsf R}(S,\beta,a,b)
:=\int_0^\infty
\sup_{u\in\mathbb S^{d-1}}
\left(n_{\beta,S}(u)-\frac t2d_{\beta,S}(u)\right)_+\dd t.
\end{equation}
Since \(d_{\beta,S}\) has a positive minimum on the sphere, the integrand vanishes for all sufficiently large \(t\).

We fix
\begin{equation}\label{eq:lambda-spectral-range}
0<\lambda<\frac12
\end{equation}
and, for \(\xi\ge0\), define
\begin{align}
q_\lambda(\xi)
&:=\xi\sup_{u\in\mathbb S^{d-1}}
\left(n_{\beta,S}(u)-\lambda \xi\,d_{\beta,S}(u)\right)_+,
\label{eq:q-spectral}\\
E_\lambda(\xi)
&:=\int_0^\xi\frac{q_\lambda(s)}s\dd s,
\label{eq:E-spectral}\\
E_\lambda^\infty
&:=\int_0^\infty
\sup_{u\in\mathbb S^{d-1}}
\left(n_{\beta,S}(u)-\lambda s\,d_{\beta,S}(u)\right)_+\dd s.
\label{eq:E-infty-spectral}
\end{align}
We understand \(q_\lambda(s)/s\) at \(s=0\) by continuous extension. It follows
that \(E_\lambda(\xi)\uparrow E_\lambda^\infty<\infty\). Let us choose
\begin{equation}\label{eq:K-spectral-choice}
K>\frac{\e^{E_\lambda^\infty}}\beta
\end{equation}
and set
\begin{equation}\label{eq:R0-spectral}
\mathscr R_0(\xi):=\frac \xi K\e^{E_\lambda(\xi)}.
\end{equation}
For \(\xi>0\), we have
\begin{equation}\label{eq:R0-log-derivative-spectral}
\frac{\xi\mathscr R_0'(\xi)}{\mathscr R_0(\xi)}
=1+q_\lambda(\xi)>0.
\end{equation}
Notice that the function \(q_\lambda\) vanishes for all sufficiently large \(\xi\), so
\(E_\lambda\) is eventually constant and \(\mathscr R_0(\xi)\to\infty\) as \(\xi\to\infty\). Along with \eqref{eq:R0-log-derivative-spectral} and \(\mathscr R_0(0)=0\), this implies
that \(\mathscr R_0\) is an increasing \(C^1\) bijection of \([0,\infty)\). We write \(p_0=\mathscr R_0^{-1}\). This inverse-penalty parametrisation follows our construction in \cite[Section 4.3 and Proposition 4.8]{Gwozdz}. We further set
\begin{equation}\label{eq:beta0-spectral}
\beta_0:=\frac{\e^{E_\lambda^\infty}}K<\beta.
\end{equation}
Since \(E_\lambda\le E_\lambda^\infty\), we obtain
\begin{equation}\label{eq:R0-beta0}
\mathscr R_0(\xi)\le\beta_0\xi.
\end{equation}
The function \(q_\lambda\) also vanishes for large \(\xi\), so \(\mathscr R_0(\xi)=\beta_0\xi\) for all sufficiently large \(\xi\). In other terms, \(p_0(r)=r/\beta_0\) for all sufficiently large \(r\).

Let us define
\begin{align}
d_0
&:=\min_{u\in\mathbb S^{d-1}}
\left(\alpha_S(u)-\beta_0^2\gamma_S(u)\right)>0,
\label{eq:d0-tail}\\
n_0
&:=\max_{u\in\mathbb S^{d-1}}
\left(\omega_b(S^{1/2}u)
+\beta_0\omega_a(S^{-1/2}u)\right)<\infty.
\label{eq:n0-tail}
\end{align}
We choose \(r_0\) sufficiently large that \(p_0(r)=r/\beta_0\) for
\(r\ge r_0\) and
\begin{equation}\label{eq:r0-tail-choice}
p_0(r_0)=\frac{r_0}{\beta_0}>\frac{2n_0}{d_0}.
\end{equation}
Let us now define \(p:[0,\infty)\to[0,\infty)\) by
\begin{equation}\label{eq:tail-completed-p}
p(r):=
\begin{cases}
p_0(r),&0\le r\le r_0,\\[1mm]
\dfrac r{\beta_0}+(r-r_0)^3,&r\ge r_0,
\end{cases}
\end{equation}
and set
\begin{equation}\label{eq:theta-Theta-spectral}
\theta(r):=\int_0^rp(s)\dd s,
\qquad
\Theta(m):=\theta(\norm m_S).
\end{equation}

For \(u,v\in\R^d\), write
\[
u\otimes v:=uv^{\mathsf T}.
\]

\begin{lemma}[Radial penalty calculus]
\label{lem:radial-penalty-calculus}
Let \(S\in\mathrm{Sym}_d^{++}\), let \(p\in C^1([0,\infty))\) satisfy \(p(0)=0\), set \(K:=p'(0)\), and define
\[
\theta(r):=\int_0^rp(s)\dd s,
\qquad
\Theta(m):=\theta(\norm m_S).
\]
It follows that \(\Theta\in C^2(\R^d)\) and \(D^2\Theta(0)=KS\). If \(m\ne0\),
\(r=\norm m_S\), and \(u=S^{1/2}m/r\), then
\begin{align}
\nabla\Theta(m)&=p(r)S^{1/2}u,
\label{eq:radial-gradient-general}\\
D^2\Theta(m)
&=p'(r)S^{1/2}u\otimes S^{1/2}u
+\frac{p(r)}rS^{1/2}(\Id-u\otimes u)S^{1/2}.
\label{eq:radial-Hessian-general}
\end{align}
If, in addition, \(K>0\) and \(p'(r)>0\) for every \(r>0\), then \(\Theta\) is strictly convex.
\end{lemma}

\begin{proof}
For \(m\ne0\), direct differentiation of \(\theta(\norm m_S)\) gives the two identities. Since \(p(r)=Kr+o(r)\) and \(p'(r)\to K\) as \(r\downarrow0\),
\eqref{eq:radial-Hessian-general} extends continuously to \(m=0\) with value
\(KS\). Assume now that \(K>0\) and \(p'>0\) on \((0,\infty)\). The radial eigenvalue \(p'(r)\) and the tangential eigenvalue \(p(r)/r\) are positive. It follows that the Hessian is positive definite for \(m\ne0\), and it is also positive definite at the origin.
\end{proof}

\begin{lemma}[Properties of the matrix-radial penalty]
\label{lem:spectral-penalty-properties}
The function \(p\) is \(C^1\) and strictly increasing. Moreover, \(\Theta\) belongs to \(C^2(\R^d)\), is strictly convex, and satisfies
\begin{equation}\label{eq:Theta-origin-spectral}
D^2\Theta(0)=KS.
\end{equation}
We also have
\begin{equation}\label{eq:Theta-superquadratic}
\frac{\Theta(m)}{\norm m^2}\longrightarrow+\infty
\qquad(\norm m\to\infty).
\end{equation}
If \(m\ne0\), \(r=\norm m_S\), and \(u=S^{1/2}m/r\), then
\begin{align}
\nabla\Theta(m)&=p(r)S^{1/2}u,
\label{eq:gradient-Theta-spectral}\\
D^2\Theta(m)
&=p'(r)S^{1/2}u\otimes S^{1/2}u
+\frac{p(r)}rS^{1/2}(\Id-u\otimes u)S^{1/2},
\label{eq:Hessian-Theta-spectral}\\
\ip{D^2\Theta(m)m}{m}&=r^2p'(r).
\label{eq:radial-Hessian-spectral}
\end{align}
\end{lemma}

\begin{proof}
By our choice of \(r_0\), the identity \(p_0(r)=r/\beta_0\) holds near \(r_0\). The two branches in \eqref{eq:tail-completed-p} have the same value and the same first derivative at \(r_0\). At the origin, \(\mathscr R_0'(0)=K^{-1}\), which gives \(p'(0)=K\). The derivative of each branch is positive, so \(p\) is strictly increasing. We now apply
\Cref{lem:radial-penalty-calculus} and obtain \eqref{eq:Theta-origin-spectral} together with \eqref{eq:gradient-Theta-spectral}--\eqref{eq:radial-Hessian-spectral}.
Finally, we observe that the cubic term in \(p\) gives the superquadratic growth.
\end{proof}

\begin{lemma}[Spectral penalty certificate]
\label{lem:spectral-penalty-certificate}
For every \(r>0\) and every \(u\in\mathbb S^{d-1}\), it holds that
\begin{align}
&\Xi_+\!\left(\frac{p(r)}{rp'(r)}\right)
+\frac12p(r)^2\alpha_S(u)
-\frac12r^2\gamma_S(u)
\notag\\
&\qquad
-p(r)\omega_b(S^{1/2}u)
-r\omega_a(S^{-1/2}u)>0.
\label{eq:spectral-penalty-certificate}
\end{align}
\end{lemma}

\begin{proof}
Let us first assume that \(r\le r_0\) and write \(q=p_0(r)\), so that \(r=\mathscr R_0(q)\). By \eqref{eq:R0-beta0}, \(r\le\beta_0q<\beta q\), while
\begin{equation}\label{eq:active-Xi-argument}
\frac{p(r)}{rp'(r)}
=\frac{q\mathscr R_0'(q)}{\mathscr R_0(q)}
=1+q_\lambda(q).
\end{equation}
The left-hand side of \eqref{eq:spectral-penalty-certificate} is bounded below by
\begin{equation}\label{eq:core-certificate-reduction}
\Xi_+(1+q_\lambda(q))
+\frac12q^2d_{\beta,S}(u)
-q\,n_{\beta,S}(u).
\end{equation}
For every \(z\ge0\),
\[
\Xi_+(1+z)\ge z,
\]
with strict inequality whenever \(z>0\). By the definition of \(q_\lambda\),
whether or not the positive part is active, we obtain
\[
q_\lambda(q)+\frac12q^2d_{\beta,S}(u)
-q\,n_{\beta,S}(u)
\ge\left(\frac12-\lambda\right)q^2d_{\beta,S}(u)>0.
\]
This establishes \eqref{eq:spectral-penalty-certificate} for \(r\le r_0\).

Let us now take \(r\ge r_0\) and put \(q=p(r)\). Since \(r\le\beta_0q\), the part of \eqref{eq:spectral-penalty-certificate} excluding \(\Xi_+\) is bounded below by
\[
\frac12q^2d_0-qn_0.
\]
The function \(p\) is increasing and \(p(r_0)>2n_0/d_0\), so this expression is
positive. With \(\Xi_+\ge0\), this shows \eqref{eq:spectral-penalty-certificate} for \(r\ge r_0\).
\end{proof}

\subsection{The smooth proof}
\label{sec:smooth}

\begin{proposition}[Smooth balanced spectral bound]
\label{prop:smooth-spectral-bound}
Let us assume that \(a,b\) are smooth and that \eqref{eq:balanced-standing} holds, and let \((S,\beta)\) be an admissible pair. It follows that
\begin{equation}\label{eq:smooth-spectral-bound}
D^2\Phi
\preceq
\frac{\exp\!\bigl(\cA_{\mathsf R}(S,\beta,a,b)\bigr)}\beta S.
\end{equation}
\end{proposition}

\begin{proof}
We apply our joint maximum-principle argument from \cite[Sections 4 and 8.1]{Gwozdz} to the translated corrector, with the matrix-radial penalty constructed in \Cref{sec:spectral-penalty}. We fix \(0<\lambda<1/2\) and \(K\) that satisfies \eqref{eq:K-spectral-choice}, and let \(\Theta\) be the penalty from
\Cref{sec:spectral-penalty}. After an additive normalisation, we may assume that \(\Phi\ge0\).

We first claim that
\begin{equation}\label{eq:G-le-Theta-spectral}
\mathfrak B_m(x)\le\Theta(m)
\qquad(x,m\in\R^d).
\end{equation}
Suppose, to the contrary, that there exist \(x_0,m_0\) and \(\eta>0\) such that
\(\mathfrak B_{m_0}(x_0)-\Theta(m_0)=\eta\). For \(\varepsilon>0\), we maximise
\begin{equation}\label{eq:J-spectral}
J_\varepsilon(x,m)
:=\mathfrak B_m(x)-\Theta(m)-\varepsilon\Phi(x).
\end{equation}
By \eqref{eq:G-global-growth}, norm equivalence, the superquadratic growth of
\(\Theta\), and the supercoercivity of \(\Phi\), the maximum is attained at some
\((x_\varepsilon,m_\varepsilon)\). For all sufficiently small \(\varepsilon\), we have
\begin{equation}\label{eq:J-positive-spectral}
J_\varepsilon(x_\varepsilon,m_\varepsilon)\ge\frac\eta2>0.
\end{equation}
In particular, \(m_\varepsilon\ne0\).

We now set
\begin{equation}\label{eq:r-u-p-spectral}
r:=r_\varepsilon:=\norm{m_\varepsilon}_S,
\qquad
u:=u_\varepsilon:=\frac{S^{1/2}m_\varepsilon}{r_\varepsilon},
\qquad
p_\varepsilon:=p(r_\varepsilon).
\end{equation}
Since \(\Phi\ge0\), the estimates \eqref{eq:J-positive-spectral} and \eqref{eq:G-global-growth}, together with norm equivalence, give
\[
\frac{\eta}{2}
\le
c_0\norm{m_\varepsilon}
+\frac{c_1}{2}\norm{m_\varepsilon}^2
-\Theta(m_\varepsilon).
\]
By \eqref{eq:Theta-superquadratic}, the right-hand side tends to zero as
\(m_\varepsilon\to0\) and to \(-\infty\) as \(\norm{m_\varepsilon}\to\infty\). This implies that \(r_\varepsilon\) remains in a fixed compact subset of \((0,\infty)\) for all sufficiently small \(\varepsilon\). The same estimates imply
\[
0\le
\varepsilon\Phi(x_\varepsilon)
\le
\mathfrak B_{m_\varepsilon}(x_\varepsilon)
-\Theta(m_\varepsilon)-\frac{\eta}{2},
\]
whose right-hand side is uniformly bounded by \eqref{eq:G-global-growth} and
the previous bound on \(r_\varepsilon\). Supercoercivity now gives
\begin{equation}\label{eq:epsilon-x-spectral}
\varepsilon\norm{x_\varepsilon}\longrightarrow0.
\end{equation}
The affine-linear modulus \eqref{eq:global-affine-modulus} then yields
\begin{equation}\label{eq:epsilon-T-spectral}
\varepsilon T(x_\varepsilon)\longrightarrow0.
\end{equation}

We apply \Cref{lem:quadratic-baseline-calculus} with \(\Gamma=0\) and
\(F^0_m=\mathfrak B_m\). At the maximising pair, set
\begin{equation}\label{eq:max-matrices-spectral}
A:=D^2\Phi(x_\varepsilon),
\qquad
C:=D^2\Phi(x_\varepsilon+m_\varepsilon),
\qquad
\Pi:=D^2\Theta(m_\varepsilon).
\end{equation}
The lemma and \eqref{eq:gradient-Theta-spectral} give
\begin{equation}\label{eq:delta-stationary-spectral}
\delta:=T(x_\varepsilon+m_\varepsilon)-T(x_\varepsilon)
=p_\varepsilon S^{1/2}u.
\end{equation}
Also, the stationarity condition in the spatial variable gives
\begin{equation}\label{eq:x-stationary-spectral}
Am_\varepsilon=\delta-\varepsilon T(x_\varepsilon).
\end{equation}
It also gives the compatibility condition
\begin{equation}\label{eq:compatibility-spectral}
(C-A)z=0\qquad(z\in\Ker(\Pi-C))
\end{equation}
and \(C\preceq\Pi\). We set
\begin{equation}\label{eq:S-at-max-spectral}
\cS_\varepsilon
:=\cS(A,C,\Pi)
=\tr\!\left[
A^{-1}(C-A)(\Pi-C)^\dagger(C-A)
\right].
\end{equation}
The trace conclusion of the lemma is
\begin{equation}\label{eq:L-Schur-spectral}
\cL\mathfrak B_m-\varepsilon\cL\Phi+\cS_\varepsilon\le0.
\end{equation}

By \Cref{lem:translated-corrector-general}, \eqref{eq:LPhi-general-bound}, and \eqref{eq:delta-stationary-spectral}, we obtain
\begin{align}
&\cH(A^{-1/2}CA^{-1/2})+\cS_\varepsilon
+\frac12p_\varepsilon^2\alpha_S(u)-\frac12r^2\gamma_S(u)
\notag\\
&\qquad
-p_\varepsilon\omega_b(S^{1/2}u)
-r\,\omega_a(S^{-1/2}u)
\le
\varepsilon\left(d+\frac{L_b^2}{4\rho}\right).
\label{eq:upper-at-max-spectral}
\end{align}

We apply \Cref{lem:matrix-coercivity} with \(P=\Pi\) and \(v=m_\varepsilon\). Recall that its compatibility hypothesis is precisely \eqref{eq:compatibility-spectral}. We define
\begin{equation}\label{eq:y-epsilon-spectral}
y_\varepsilon
:=\frac{\ip{Am_\varepsilon}{m_\varepsilon}}
{\ip{\Pi m_\varepsilon}{m_\varepsilon}},
\end{equation}
and the lemma gives
\begin{equation}\label{eq:lower-at-max-spectral}
\cH(A^{-1/2}CA^{-1/2})+\cS_\varepsilon
\ge\Xi_+(y_\varepsilon).
\end{equation}
By \eqref{eq:radial-Hessian-spectral},
\[
\ip{\Pi m_\varepsilon}{m_\varepsilon}=r^2p'(r).
\]
We take the scalar product of \eqref{eq:x-stationary-spectral} with
\(m_\varepsilon=rS^{-1/2}u\) to obtain
\[
\ip{Am_\varepsilon}{m_\varepsilon}
=rp_\varepsilon-\varepsilon\ip{T(x_\varepsilon)}{m_\varepsilon}.
\]
After passing to a subsequence, we may assume that
\[
r_\varepsilon\to r_*>0,
\qquad
u_\varepsilon\to u_*\in\mathbb S^{d-1}.
\]
It follows from \eqref{eq:epsilon-T-spectral} that
\begin{equation}\label{eq:y-limit-spectral}
y_\varepsilon\longrightarrow\frac{p(r_*)}{r_*p'(r_*)}.
\end{equation}
It suffices to combine \eqref{eq:upper-at-max-spectral} and \eqref{eq:lower-at-max-spectral} to get
\begin{align*}
&\Xi_+(y_\varepsilon)
+\frac12p(r_\varepsilon)^2\alpha_S(u_\varepsilon)
-\frac12r_\varepsilon^2\gamma_S(u_\varepsilon)\\
&\qquad
-p(r_\varepsilon)\omega_b(S^{1/2}u_\varepsilon)
-r_\varepsilon\omega_a(S^{-1/2}u_\varepsilon)
\le
\varepsilon\left(d+\frac{L_b^2}{4\rho}\right).
\end{align*}
We pass to the subsequence above and use \eqref{eq:y-limit-spectral}, which
gives
\begin{align*}
&\Xi_+\!\left(\frac{p(r_*)}{r_*p'(r_*)}\right)
+\frac12p(r_*)^2\alpha_S(u_*)
-\frac12r_*^2\gamma_S(u_*)\\
&\qquad
-p(r_*)\omega_b(S^{1/2}u_*)
-r_*\omega_a(S^{-1/2}u_*)
\le0,
\end{align*}
but this contradicts \Cref{lem:spectral-penalty-certificate}, so we conclude that
\eqref{eq:G-le-Theta-spectral} holds.

We fix \(x,h\in\R^d\). As \(t\to0\),
\[
\mathfrak B_{th}(x)
=\frac{t^2}{2}\ip{D^2\Phi(x)h}{h}+o(t^2),
\]
while \eqref{eq:Theta-origin-spectral} gives
\[
\Theta(th)=\frac{Kt^2}{2}\ip{Sh}{h}+o(t^2).
\]
It follows from the two expansions that \(D^2\Phi\preceq KS\). We first let
\(K\downarrow\beta^{-1}\e^{E_\lambda^\infty}\) and then let \(\lambda\uparrow1/2\). We also fix \(\lambda_0\in(0,1/2)\). For \(\lambda\in[\lambda_0,1/2)\), the integrand in
\eqref{eq:E-infty-spectral} is bounded by the integrable integrand associated
with \(\lambda_0\). By dominated convergence,
\[
E_\lambda^\infty
\downarrow
\cA_{\mathsf R}(S,\beta,a,b)
\qquad\text{as }\lambda\uparrow\frac12.
\]
This proves the required estimate.
\end{proof}

\subsection{Approximation and original coordinates}
\label{sec:general-proof}

We approximate the source by a mollified Moreau envelope with a compensating
constant, and we analyse the target by direct mollification. Both approximations preserve the required one-sided curvature bounds and give integrable dominating functions. The source construction refines our result from \cite[proof of Lemma 7.1]{Gwozdz}.

We first work in balanced coordinates and assume that \(V,W\) are finite
continuous functions and
\begin{equation}\label{eq:balanced-distributional-standing}
D^2V\preceq\mathsf R,
\qquad
D^2W\succeq\mathsf R
\end{equation}
in the sense of distributions. Let us set
\[
q_{\mathsf R}(x):=\frac12\ip{\mathsf Rx}{x},
\qquad
f:=q_{\mathsf R}-V,
\qquad
g:=W-q_{\mathsf R}.
\]
By the distributional characterisation of convexity,
\eqref{eq:balanced-distributional-standing} implies that the continuous functions \(f\) and \(g\) are finite and convex. Let \(\zeta\in C_c^\infty(\R^d)\) be nonnegative and symmetric, with \(\int_{\R^d}\zeta(x)\,\dd x=1\), and, for \(\delta>0\), write
\[
\zeta_\delta(x):=\delta^{-d}\zeta(x/\delta),
\qquad
m_2:=\int_{\R^d}\norm z^2\zeta(z)\dd z.
\]
We further choose \(\varepsilon_n,\delta_n\downarrow0\) such that
\(\delta_n^2/\varepsilon_n\to0\). For \(\varepsilon>0\), let
\[
f_\varepsilon(x)
:=
\inf_{z\in\R^d}
\left\{
f(z)+\frac{\norm{x-z}^2}{2\varepsilon}
\right\}
\]
be the Moreau envelope of \(f\) \cite{Moreau}. Since \(f\) is finite and convex,
\(f_\varepsilon\) is finite and convex, belongs to \(C^{1,1}(\R^d)\), and satisfies
\[
0\preceq D^2f_\varepsilon
\preceq\varepsilon^{-1}\Id
\]
in the sense of distributions. We define
\begin{equation}\label{eq:one-sided-source-approx}
\widehat f_n
:=
f_{\varepsilon_n}*\zeta_{\delta_n}
-\frac{m_2\delta_n^2}{2\varepsilon_n},
\qquad
V_n:=q_{\mathsf R}-\widehat f_n.
\end{equation}
Observe that the function \(\widehat f_n\) is smooth and convex. Since
\(\nabla f_{\varepsilon_n}\) is \(\varepsilon_n^{-1}\)-Lipschitz and \(\zeta\) is
symmetric, we obtain
\[
f_{\varepsilon_n}*\zeta_{\delta_n}
\le f_{\varepsilon_n}
+\frac{m_2\delta_n^2}{2\varepsilon_n},
\]
and the mollifier \(\zeta\) is centred. Jensen's inequality and the previous
smoothness estimate give
\[
f_{\varepsilon_n}
-\frac{m_2\delta_n^2}{2\varepsilon_n}
\le
\widehat f_n
\le
f_{\varepsilon_n}.
\]
Since \(\delta_n^2/\varepsilon_n\to0\) and \(f_{\varepsilon_n}\to f\) pointwise, we obtain
\(\widehat f_n\to f\) pointwise. We also have \(\widehat f_n\le f_{\varepsilon_n}\le f\), so
\begin{equation}\label{eq:source-approx-properties}
V_n\ge V,
\qquad
D^2V_n\preceq\mathsf R.
\end{equation}
For the target, set
\begin{equation}\label{eq:one-sided-target-approx}
W_n:=q_{\mathsf R}+g*\zeta_{\delta_n}.
\end{equation}
Jensen's inequality gives \(g*\zeta_{\delta_n}\ge g\). We have
\begin{equation}\label{eq:target-approx-properties}
W_n\ge W,
\qquad
D^2W_n\succeq\mathsf R.
\end{equation}

We next prove convergence of the approximating marginals. To this end, set
\[
a_n:=a*\zeta_{\delta_n},
\qquad
b_n:=b*\zeta_{\delta_n}.
\]
We define the normalised approximating marginals by
\[
Z_{\mu,n}:=\int_{\R^d}\e^{-V_n(x)-a_n(x)}\dd x,
\qquad
Z_{\nu,n}:=\int_{\R^d}\e^{-W_n(y)-b_n(y)}\dd y,
\]
\[
\dd\mu_n(x):=Z_{\mu,n}^{-1}\e^{-V_n(x)-a_n(x)}\dd x,
\qquad
\dd\nu_n(y):=Z_{\nu,n}^{-1}\e^{-W_n(y)-b_n(y)}\dd y.
\]
By construction,
\begin{equation}\label{eq:mollified-width-control}
\omega_{a_n}\le\omega_a,
\qquad
\omega_{b_n}\le\omega_b.
\end{equation}
Moreover, the approximations satisfy \(V_n\to V\) and \(W_n\to W\) pointwise, while \(a_n\to a\) and \(b_n\to b\) uniformly. There also exists a constant \(C<\infty\),
independent of \(n\), such that
\[
(1+\norm x^2)\e^{-V_n(x)-a_n(x)}
\le
C(1+\norm x^2)\e^{-V(x)-a(x)}
\]
and
\[
(1+\norm y^2)\e^{-W_n(y)-b_n(y)}
\le
C(1+\norm y^2)\e^{-W(y)-b(y)}.
\]
These inequalities follow from \(V_n\ge V\), \(W_n\ge W\), and the uniform
convergence of \(a_n\) and \(b_n\). Dominated convergence gives convergence of
the normalising constants and of the normalised densities in weighted \(L^1\). Let \(Z_{\mu,a}\) and \(Z_{\nu,b}\) denote the limiting source and target normalising constants, respectively. In particular,
\[
\int_{\R^d}(1+\norm x^2)
\left|
\frac{\e^{-V_n-a_n}}{Z_{\mu,n}}
-
\frac{\e^{-V-a}}{Z_{\mu,a}}
\right|\dd x
\longrightarrow0.
\]
Note that the target integral also tends to zero. We denote
\[
r_{\mathrm{s},n}(z)
:=\frac{\e^{-V_n(z)-a_n(z)}}{Z_{\mu,n}},
\qquad
r_{\mathrm{s}}(z)
:=\frac{\e^{-V(z)-a(z)}}{Z_{\mu,a}},
\]
and
\[
r_{\mathrm{t},n}(z)
:=\frac{\e^{-W_n(z)-b_n(z)}}{Z_{\nu,n}},
\qquad
r_{\mathrm{t}}(z)
:=\frac{\e^{-W(z)-b(z)}}{Z_{\nu,b}}.
\]
For \(\iota\in\{\mathrm{s},\mathrm{t}\}\), we couple the common part of
\(r_{\iota,n}\) and \(r_\iota\) identically, and the two residual parts arbitrarily. This coupling yields
\[
W_2^2(r_{\iota,n}\dd z,r_\iota\dd z)
\le
2\int_{\R^d}\norm z^2
|r_{\iota,n}(z)-r_\iota(z)|\dd z.
\]
The above estimates imply that the approximating marginals converge both
in total variation and in \(W_2\).

We now pass to the Brenier representatives. Let us fix an admissible pair
\((S,\beta)\) and normalise the Brenier potentials for the smooth approximating data by \(\Phi_n(0)=0\). We combine \Cref{prop:smooth-spectral-bound} with
\eqref{eq:mollified-width-control} to obtain
\begin{equation}\label{eq:Hessian-n-spectral}
K_{S,\beta}:=\frac{\e^{\cA_{\mathsf R}(S,\beta,a,b)}}{\beta},
\qquad
0\preceq D^2\Phi_n
\preceq
K_{S,\beta}S.
\end{equation}

In particular, the Euclidean Lipschitz constants of \(T_n=\nabla\Phi_n\) are
uniformly bounded by \(K_{S,\beta}\opnorm S\). We also have
\[
\norm{T_n(0)}
\le\int\norm y\,\dd\nu_n(y)
+K_{S,\beta}\opnorm S\int\norm x\,\dd\mu_n(x),
\]
so \(T_n(0)\) is uniformly bounded, and the stability lemma
\Cref{lem:uniform-Lipschitz-stability} applies. Since the limiting source density is strictly positive on \(\R^d\), the lemma gives
\[
T_n\longrightarrow T
\qquad\text{and}\qquad
\Phi_n\longrightarrow\Phi
\]
locally uniformly. Here, \(T=\nabla\Phi\) is the globally continuous Brenier
representative from \(\mu_a\) to \(\nu_b\), and \(\Phi(0)=0\).

It remains to pass the Hessian bound to the limit. Observe that the matrix inequality
\eqref{eq:Hessian-n-spectral} is equivalent to
\[
\Phi_n(x+h)+\Phi_n(x-h)-2\Phi_n(x)
\le K_{S,\beta}\ip{Sh}{h}.
\]
We pass to the locally uniform limit and apply \Cref{lem:C11}. It follows that
\begin{equation}\label{eq:balanced-final-spectral}
D^2\Phi\preceq K_{S,\beta}S.
\end{equation}

From this point, we again use \(V,W,a,b,T,\Phi\) and the associated marginals for the original-coordinate objects, and we keep tildes for their counterparts in balanced coordinates.

Let \(H\succ0\) satisfy \(H\mathsf RH\succ\mathsf R\). We choose \(0<\beta<1\) and set \(S=\beta H\), so that \((S,\beta)\) is admissible. By \eqref{eq:balanced-widths}, the positive homogeneity of the width gauges, and the change of variables \(\tau=t\sqrt\beta\), we have
\[
d_{\beta,S}(u)=\beta\ip{\mathsf D_Hu}{u},
\qquad
\omega_{\widetilde b}(S^{1/2}u)+\beta\omega_{\widetilde a}(S^{-1/2}u)
=\sqrt\beta\,n_{Q,P}^{a,b,\mathrm{red}}(H,u),
\]
and obtain
\[
\cA_{\mathsf R}(S,\beta,\widetilde a,\widetilde b)
=\cA_{Q,P}^{\mathrm{red}}(H,a,b).
\]
The balanced estimate \eqref{eq:balanced-final-spectral} now gives
\[
D^2\widetilde\Phi
\preceq
\e^{\cA_{Q,P}^{\mathrm{red}}(H,a,b)}H.
\]
Using \eqref{eq:matrix-scaled-map}, we obtain
\[
D^2\Phi
\preceq
\e^{\cA_{Q,P}^{\mathrm{red}}(H,a,b)}
G^{1/2}HG^{1/2}.
\]
Let us set \(M:=M_H\), so that
\[
0\preceq D^2\Phi\preceq M
\]
in the sense of matrix-valued distributions and almost everywhere. By
\Cref{lem:C11}, we also have \(\Phi\in C^{1,1}(\R^d)\). Finally, the metric estimate \eqref{eq:spectral-metric-map-main} follows from \Cref{lem:hessian-bound-metric}. We conclude this section by deriving the one-sided corollaries from \Cref{sec:geometry}.

\begin{proof}[Proof of \Cref{cor:affine-sharpness}]
Recall that affine functions have zero gradient width. We take \(H=c\Id\) with \(c>1\)
in the reduced description. In this case, the action vanishes and the estimate becomes
\(D^2\Phi\preceq cG\). We now let \(c\downarrow1\) and obtain \eqref{eq:unperturbed-matrix-sharp}. For the Gaussian pair, observe that the map \(T(x)=Gx\) sends \(\mathcal N(0,Q^{-1})\) to \(\mathcal N(0,P^{-1})\), since \(GQ^{-1}G=P^{-1}\). We conclude that
\(D^2\Phi=G\).
\end{proof}

\begin{proof}[Proof of \Cref{cor:matrix-caffarelli}]
We take \(H=\beta^{-1}\Id\), where \(0<\beta<1\), in \Cref{thm:spectral-caffarelli}. It holds that
\[
\mathsf D_H=(\beta^{-1}-\beta)\mathsf R,
\qquad
n_{Q,P}^{a,b,\mathrm{red}}(H,u)
\le2\beta^{-1/2}\ell_b+2\beta^{1/2}\ell_a.
\]
The definitions give
\[
\cA_{Q,P}^{\mathrm{red}}(H,a,b)
\le
\frac{4(\ell_b+\beta\ell_a)^2}{\rho(1-\beta^2)},
\]
and the majorant is \(\beta^{-1}\exp(\cA_{Q,P}^{\mathrm{red}}(H,a,b))G\). It now suffices to take the infimum over \(\beta\) to finish the claim.
\end{proof}

\begin{proof}[Proof of \Cref{cor:perturbative-caffarelli}]
Let us take \(Q=\Lambda\Id\) and \(P=\kappa\Id\), then
\[
G=\sqrt{\Lambda/\kappa}\,\Id,
\qquad
\mathsf R=\sqrt{\Lambda\kappa}\,\Id,
\qquad
\rho=\sqrt{\Lambda\kappa}.
\]
We also have
\[
\Lip_G(a)=\frac{\Lip(a)}{(\Lambda/\kappa)^{1/4}},
\qquad
\Lip_{G^{-1}}(b)=\Lip(b)(\Lambda/\kappa)^{1/4}.
\]
We substitute these identities into \eqref{eq:def-Crho}, and obtain \eqref{eq:perturbative-upper-main}.
\end{proof}

\section{Two-sided estimates and Gaussian perturbations}
\label{sec:two-sided-gaussian}

In this section, we assume that \(V,W:\R^d\to\R\) are finite and continuous, that \(a,b:\R^d\to\R\) are globally Lipschitz, and that the measures in \eqref{eq:matrix-marginals-main} are probability measures with finite second moments. Let \(T=\nabla\Phi\) denote their quadratic-cost Brenier map.

\begin{lemma}[Forward--reverse transfer]
\label{lem:forward-reverse-transfer}
Under the hypotheses from this section, let \(U=\nabla\Psi\) be the globally continuous Brenier representative from \(\nu_b\) to \(\mu_a\). We assume that \(\Phi,\Psi\in C^{1,1}(\R^d)\) and that, for some \(M_+,M_-\in\mathrm{Sym}_d^{++}\),
\[
0\preceq D^2\Phi\preceq M_+,
\qquad
0\preceq D^2\Psi\preceq M_-.
\]
It follows that
\begin{equation}\label{eq:transfer-Hessian}
M_-^{-1}\preceq D^2\Phi\preceq M_+,
\end{equation}
and
\begin{align}
\norm{T(x)-T(y)}_{M_+^{-1}}
&\le\norm{x-y}_{M_+},
\label{eq:transfer-upper-map}\\
\norm{T(x)-T(y)}_{M_-}
&\ge\norm{x-y}_{M_-^{-1}}
\label{eq:transfer-lower-map}
\end{align}
for all \(x,y\in\R^d\).
\end{lemma}

\begin{proof}
The uniqueness of the two optimal plans gives
\[
U\circ T=\Id
\quad\text{\(\mu_a\)-almost everywhere},
\qquad
T\circ U=\Id
\quad\text{\(\nu_b\)-almost everywhere}.
\]
Given that both marginal densities are strictly positive on \(\R^d\), every nonempty open set has positive measure for the corresponding marginal. Since \(T\) and \(U\) are continuous, the two identities hold everywhere. We conclude that \(T\) is a homeomorphism with inverse \(U\).

As \(T(\R^d)=\R^d\) holds, \Cref{lem:supercoercive} implies that \(\Phi\) is
supercoercive. In particular, \(\Phi^*\) is finite on \(\R^d\). The subgradient-inversion identity now gives \(\nabla\Phi^*=U=\nabla\Psi\). After changing \(\Psi\) by an additive constant, we take \(\Psi=\Phi^*\). We fix \(x_0\in\R^d\) and put \(y_0=\nabla\Phi(x_0)\). Since \(x_0=\nabla\Psi(y_0)\) and \(D^2\Psi\preceq M_-\), the upper Hessian bound gives
\[
\Psi(y)
\le
\Psi(y_0)+\ip{x_0}{y-y_0}
+\frac12\ip{M_-(y-y_0)}{y-y_0}
\qquad(y\in\R^d).
\]
We use this estimate in \(\Phi=\Psi^*\) and obtain
\[
\Phi(x)\ge
\Phi(x_0)+\ip{y_0}{x-x_0}
+\frac12\ip{M_-^{-1}(x-x_0)}{x-x_0}.
\]
It follows that \(x\mapsto\Phi(x)-\ip{M_-^{-1}x}{x}/2\) is convex, which proves \eqref{eq:transfer-Hessian}.

The upper metric estimate follows by applying \Cref{lem:hessian-bound-metric} to \(\Phi\) and \(M_+\). On the other hand, the lower Hessian bound gives
\[
\ip{T(x)-T(y)}{x-y}\ge\norm{x-y}_{M_-^{-1}}^2,
\]
and dual Cauchy--Schwarz yields \eqref{eq:transfer-lower-map}, which completes the proof.
\end{proof}

\begin{corollary}[Spectral bi-Lipschitz estimate]
\label{cor:spectral-bilip}
Let us fix \(P_0,Q_0,P_1,Q_1\in\mathrm{Sym}_d^{++}\), and assume that
\[
P_0\preceq D^2V\preceq Q_0,
\qquad
P_1\preceq D^2W\preceq Q_1.
\]
Let \(G_+,G_-\succ0\) be the solutions of
\[
G_+P_1G_+=Q_0,
\qquad
G_-P_0G_-=Q_1,
\]
and set
\[
\mathsf R_+:=G_+^{-1/2}Q_0G_+^{-1/2},
\qquad
\mathsf R_-:=G_-^{-1/2}Q_1G_-^{-1/2}.
\]
For the forward data \((Q_0,P_1,a,b)\), we choose a reduced metric \(H_+\in\mathrm{Sym}_d^{++}\) such that \(H_+\mathsf R_+H_+\succ\mathsf R_+\), and let \(M_+\) be the corresponding matrix from \eqref{eq:M-H-main}. For the reverse data \((Q_1,P_0,b,a)\), choose a reduced metric \(H_-\in\mathrm{Sym}_d^{++}\) such that
\(H_-\mathsf R_-H_-\succ\mathsf R_-\), and define \(M_-\) in the same way. We
then obtain
\begin{equation}\label{eq:spectral-bilip-Hessian}
M_-^{-1}\preceq D^2\Phi\preceq M_+
\end{equation}
in distributions and almost everywhere. Moreover,
\begin{align}
\norm{T(x)-T(y)}_{M_+^{-1}}
&\le\norm{x-y}_{M_+}
\qquad(x,y\in\R^d),
\label{eq:spectral-bilip-upper-map}\\
\norm{T(x)-T(y)}_{M_-}
&\ge\norm{x-y}_{M_-^{-1}}
\qquad(x,y\in\R^d).
\label{eq:spectral-bilip-lower-map}
\end{align}
\end{corollary}

For \(0<\beta_+,\beta_-<1\), observe that the choices \(H_+=\beta_+^{-1}\Id\) and \(H_-=\beta_-^{-1}\Id\) in the forward and reverse reduced problems correspond to \(S_+=S_-=\Id\) in the balanced parametrisation.

\begin{corollary}[Ellipsoidal bi-Lipschitz estimate]
\label{cor:matrix-bilip}
Under the two-sided curvature hypotheses of \Cref{cor:spectral-bilip}, let
\(G_\pm\) and \(\mathsf R_\pm\) be as above, and define
\[
\rho_+:=\lambda_{\min}(\mathsf R_+),
\quad
c_+:=\cC_{\rho_+}
\bigl(\Lip_{G_+}(a),\Lip_{G_+^{-1}}(b)\bigr).
\]
We also set
\[
\rho_-:=\lambda_{\min}(\mathsf R_-),
\quad
c_-:=\cC_{\rho_-}
\bigl(\Lip_{G_-}(b),\Lip_{G_-^{-1}}(a)\bigr).
\]
It follows that
\begin{equation}\label{eq:matrix-bilip-Hessian}
c_-^{-1}G_-^{-1}\preceq D^2\Phi\preceq c_+G_+,
\end{equation}
and
\begin{align}
\norm{T(x)-T(y)}_{G_+^{-1}}
&\le c_+\norm{x-y}_{G_+}
\qquad(x,y\in\R^d),
\label{eq:matrix-bilip-upper-map}\\
\norm{T(x)-T(y)}_{G_-}
&\ge c_-^{-1}\norm{x-y}_{G_-^{-1}}
\qquad(x,y\in\R^d).
\label{eq:matrix-bilip-lower-map}
\end{align}
\end{corollary}

\begin{corollary}[Scalar perturbative bi-Lipschitz estimate]
\label{cor:perturbative-bilip}
Let \(0<\lambda_V\le\Lambda_V\), \(0<\lambda_W\le\Lambda_W\), and assume that
\[
\lambda_V\Id\preceq D^2V\preceq\Lambda_V\Id,
\qquad
\lambda_W\Id\preceq D^2W\preceq\Lambda_W\Id.
\]
We set
\begin{align*}
C_{\rightarrow}
&:={\sqrt{\Lambda_V/\lambda_W}}\,
\cC\!\left(
\frac{\Lip(a)}{\sqrt{\Lambda_V}},
\frac{\Lip(b)}{\sqrt{\lambda_W}}
\right),\\
C_{\leftarrow}
&:={\sqrt{\Lambda_W/\lambda_V}}\,
\cC\!\left(
\frac{\Lip(b)}{\sqrt{\Lambda_W}},
\frac{\Lip(a)}{\sqrt{\lambda_V}}
\right).
\end{align*}
We then have
\begin{equation}\label{eq:perturbative-bilip-Hessian}
C_{\leftarrow}^{-1}\Id\preceq D^2\Phi\preceq C_{\rightarrow}\Id,
\end{equation}
and
\begin{equation}\label{eq:perturbative-bilip-map}
C_{\leftarrow}^{-1}\norm{x-y}
\le\norm{T(x)-T(y)}
\le C_{\rightarrow}\norm{x-y}
\qquad(x,y\in\R^d).
\end{equation}
\end{corollary}

For a globally Lipschitz function \(a:\R^d\to\R\), let us write
\begin{equation}\label{eq:def-gamma-a}
Z_a:=\int_{\R^d}\e^{-a(x)}\dd\gamma_d(x),
\qquad
\dd\gamma_a(x)
:=Z_a^{-1}\e^{-a(x)}\dd\gamma_d(x).
\end{equation}

\begin{corollary}[Two log-Lipschitz Gaussian perturbations]
\label{cor:gaussian-bilip}
Let \(a,b:\R^d\to\R\) be globally Lipschitz, and denote by \(T_{a,b}=\nabla\Phi\) the Brenier map from \(\gamma_a\) to \(\gamma_b\). It follows that
\begin{equation}\label{eq:gaussian-bilip-Hessian}
\cC(\Lip(b),\Lip(a))^{-1}\Id
\preceq D^2\Phi
\preceq\cC(\Lip(a),\Lip(b))\Id,
\end{equation}
and
\begin{equation}\label{eq:gaussian-bilip-map}
\frac1{\cC(\Lip(b),\Lip(a))}\norm{x-y}
\le\norm{T_{a,b}(x)-T_{a,b}(y)}
\le\cC(\Lip(a),\Lip(b))\norm{x-y}
\qquad(x,y\in\R^d).
\end{equation}
\end{corollary}

\begin{proof}[Proof of \Cref{cor:spectral-bilip}]
The application of \Cref{thm:spectral-caffarelli} to the forward data gives
\(D^2\Phi\preceq M_+\). Let \(U=\nabla\Psi\) denote the reverse Brenier
representative. The same theorem, which we apply to the reverse data
\((Q_1,P_0,b,a)\), gives \(D^2\Psi\preceq M_-\). The conclusions now follow from
\Cref{lem:forward-reverse-transfer}.
\end{proof}

\begin{proof}[Proofs of \Cref{cor:matrix-bilip,cor:perturbative-bilip}]
For \Cref{cor:matrix-bilip}, we apply \Cref{cor:matrix-caffarelli} to the forward and reverse maps. We use \Cref{lem:forward-reverse-transfer} with \(M_+=c_+G_+\) and \(M_-=c_-G_-\). The argument for \Cref{cor:perturbative-bilip} is the same, where we use \Cref{cor:perturbative-caffarelli}, \(M_+=C_{\rightarrow}\Id\), and \(M_-=C_{\leftarrow}\Id\).
\end{proof}

\begin{proof}[Proof of \Cref{cor:gaussian-bilip}]
It suffices to take
\[
V(x)=W(x)=\frac12\norm{x}^2
\]
in \Cref{cor:perturbative-bilip}, with all four curvature constants equal to one.
\end{proof}

\subsection{Gaussian constants and displacement}
\label{sec:gaussian-applications}

\begin{proposition}[Optimisation of the Gaussian constants]
\label{prop:constant-optimization}
We fix \(\ell_a,\ell_b\ge0\) with \(\ell_a+\ell_b>0\). For \(\rho=1\), the infimum in
\eqref{eq:def-Crho} is attained at a unique point \(\beta_*\in(0,1)\). This point is characterised by
\begin{equation}\label{eq:beta-star-equation}
(1-\beta_*^2)^2
=
8\beta_*(\ell_b+\ell_a\beta_*)(\ell_a+\ell_b\beta_*).
\end{equation}
Moreover, for \(L>0\),
\begin{align}
\cC(0,L)
&=
\frac{2L+\sqrt{4L^2+2}}{\sqrt2}
\exp\!\left(2L^2+L\sqrt{4L^2+2}\right),
\label{eq:Cplus-closed}\\
\cC(L,0)
&=
\frac{2L+\sqrt{4L^2+2}}{\sqrt2}
\exp\!\left(L\sqrt{4L^2+2}-2L^2\right).
\label{eq:Cminus-closed}
\end{align}
We further have
\begin{equation}\label{eq:C-asymptotics}
\cC(0,L)\sim2\sqrt{2\e}\,L\e^{4L^2},
\qquad
\cC(L,0)\sim2\sqrt{2\e}\,L
\qquad(L\to\infty).
\end{equation}
\end{proposition}

\begin{proof}
Let us assume that \(\ell_a+\ell_b>0\), and set
\[
F(\beta)
:=
-\log\beta
+
\frac{4(\ell_b+\ell_a\beta)^2}{1-\beta^2},
\qquad0<\beta<1.
\]
Observe that the function satisfies \(F(\beta)\to+\infty\) at both endpoints, while
\[
F'(\beta)
=
-\frac1\beta
+
\frac{8(\ell_b+\ell_a\beta)(\ell_a+\ell_b\beta)}
{(1-\beta^2)^2}.
\]
Notice that the function
\[
\beta\longmapsto
\frac{8\beta(\ell_b+\ell_a\beta)(\ell_a+\ell_b\beta)}
{(1-\beta^2)^2}
\]
is strictly increasing from \(0\) to \(+\infty\), which implies that the function \(F\) has
a unique critical point. This point is its unique minimiser, and it is characterised by \eqref{eq:beta-star-equation}.

Let us first take \(\ell_a=0\) and \(\ell_b=L\), and write \(\lambda=(1-\beta^2)/2\). The minimisation problem becomes
\[
\inf_{0<\lambda<1/2}
\frac{\exp(2L^2/\lambda)}{\sqrt{1-2\lambda}}.
\]
The logarithmic derivative vanishes at
\[
\lambda_*
=
L\bigl(\sqrt{4L^2+2}-2L\bigr),
\]
and substitution gives \eqref{eq:Cplus-closed}.

We next take \(\ell_a=L\) and \(\ell_b=0\). The equation for the minimiser now
reads
\[
1-\beta^2=2\sqrt2L\beta,
\]
so
\[
\beta_*=\sqrt{2L^2+1}-\sqrt2L.
\]
Inserting this value gives \eqref{eq:Cminus-closed}, and direct expansion
yields the two asymptotic identities.
\end{proof}

\begin{proof}[Proof of \Cref{thm:gaussian-source}]
It suffices to \Cref{cor:gaussian-bilip} with \(a=0\) and \(b=B\). The conclusion then
follows directly from \Cref{prop:constant-optimization}.
\end{proof}

\begin{remark}[An obstruction]
\label{rem:neeman-obstruction}
We fix \(L>0\), and set
\[
B^{(L)}(t):=-L|t|,
\qquad
Z_L:=\int_{\R}\e^{L|t|}\dd\gamma_1(t),
\qquad
\dd\nu_L(t):=Z_L^{-1}\e^{L|t|}\dd\gamma_1(t),
\]
and let \(T_L\) denote the monotone transport from \(\gamma_1\) to \(\nu_L\). We write
\[
\varphi_1(t):=(2\pi)^{-1/2}\e^{-t^2/2},
\]
for the standard Gaussian density. By symmetry, \(T_L(0)=0\). We differentiate the identity between the source and target distribution functions at the origin, and obtain
\[
T_L'(0)
=
\frac{\varphi_1(0)}{Z_L^{-1}\varphi_1(0)}
=Z_L.
\]
By direct computation, we have
\[
Z_L
=
2\int_0^\infty \e^{Lt}\varphi_1(t)\dd t
=
2\e^{L^2/2}\int_{-L}^\infty\varphi_1(s)\dd s
\ge \e^{L^2/2},
\]
and conclude that
\[
\Lip(T_L)\ge T_L'(0)\ge\e^{L^2/2}.
\]
This is the normalising-constant obstruction from \cite[Remark 1]{FMS}. In that remark, the formula \(W(x)=L|x|+\log Z\) is inconsistent with the identity \(Z=\int \e^{L|x|}\,\dd\mu\), and the calculation uses \(W(x)=-L|x|+\log Z\). In our convention
\[
\dd\nu_B=Z_B^{-1}\e^{-B}\dd\gamma_1,
\]
the heavy-tailed perturbation is \(B^{(L)}(t)=-L|t|\).
\end{remark}

\begin{theorem}[Displacement for Gaussian perturbations]
\label{thm:gaussian-displacement}
Let \(a,b:\R^d\to\R\) be globally Lipschitz, and let \(T_{a,b}\) be the globally continuous Brenier representative from \(\gamma_a\) to \(\gamma_b\). It follows that
\begin{equation}\label{eq:gaussian-displacement-body}
T_{a,b}(x)-x\in\cK_a-\cK_b
\qquad(x\in\R^d).
\end{equation}
In particular,
\begin{equation}\label{eq:gaussian-displacement-norm}
\sup_{x\in\R^d}\norm{T_{a,b}(x)-x}
\le
\sup_{p\in\cK_a,\,q\in\cK_b}\norm{p-q}
\le\Lip(a)+\Lip(b).
\end{equation}
Notice that the inclusion is exact for affine tilts. More precisely, let \(p,q\in\R^d\), \(c_a,c_b\in\R\), \(a(x)=\ip{p}{x}+c_a\), and \(b(y)=\ip{q}{y}+c_b\). One then has \(T_{a,b}(x)=x+p-q\) for every \(x\in\R^d\).
\end{theorem}

\begin{proof}
We begin with smooth \(a\) and \(b\), write \(T_{a,b}=\nabla\Phi\), and set
\[
A:=D^2\Phi,
\qquad
S:=T_{a,b}-\Id.
\]
For \(f\in C^2(\R^d)\), we define
\begin{equation}\label{eq:displacement-operator-two}
(\cL f)(x)
:=\tr\!\left(A(x)^{-1}D^2f(x)\right)
-\ip{T_{a,b}(x)+\nabla b(T_{a,b}(x))}{\nabla f(x)}.
\end{equation}
Using this notation, the Monge--Amp\`ere equation is
\[
\log\det A
=c-\frac12\norm{x}^2-a(x)
+\frac12\norm{T_{a,b}(x)}^2+b(T_{a,b}(x)),
\qquad c\in\R.
\]
We fix \(e\in\mathbb S^{d-1}\), and write \(a_e:=\ip{\nabla a}{e}\), \(b_e:=\ip{\nabla b}{e}\), and \(S_e:=\ip{S}{e}\). Differentiating once gives
\begin{equation}\label{eq:displacement-linear-equation-two}
\cL S_e
=S_e+b_e(T_{a,b}(x))-a_e(x).
\end{equation}
The same equation also gives
\begin{equation}\label{eq:LPhi-displacement-two}
\cL\Phi
=d-\ip{T_{a,b}(x)+\nabla b(T_{a,b}(x))}{T_{a,b}(x)}
\le d+\frac{\Lip(b)^2}{4}.
\end{equation}

By \Cref{prop:smooth-diffeo-general,lem:supercoercive}, the potential is supercoercive. After adding a constant, we assume that \(\Phi\ge0\). The estimate in \Cref{cor:gaussian-bilip} shows that \(S\) has at most linear growth. For every \(\varepsilon>0\), the function \(S_e-\varepsilon\Phi\) attains a global maximum at some
\(x_\varepsilon\). At this point,
\[
\nabla(S_e-\varepsilon\Phi)(x_\varepsilon)=0,
\qquad
D^2(S_e-\varepsilon\Phi)(x_\varepsilon)\preceq0,
\]
so
\[
\cL(S_e-\varepsilon\Phi)(x_\varepsilon)\le0.
\]
We combine \eqref{eq:displacement-linear-equation-two} and \eqref{eq:LPhi-displacement-two}, and obtain
\[
S_e(x_\varepsilon)
\le a_e(x_\varepsilon)-b_e(T_{a,b}(x_\varepsilon))
+\varepsilon\left(d+\frac{\Lip(b)^2}{4}\right).
\]
We now fix \(x_0\). By maximality and the inequality \(\Phi\ge0\),
\[
S_e(x_0)-\varepsilon\Phi(x_0)
\le S_e(x_\varepsilon)-\varepsilon\Phi(x_\varepsilon)
\le S_e(x_\varepsilon).
\]
We now let \(\varepsilon\downarrow0\) and obtain
\begin{equation}\label{eq:displacement-support-bound-two}
\ip{S(x_0)}{e}
\le
\sup_{z\in\R^d}a_e(z)-\inf_{z\in\R^d}b_e(z)
=h_{\cK_a-\cK_b}(e).
\end{equation}
As the estimate holds for every unit vector \(e\), the support-function characterisation of closed convex sets yields \eqref{eq:gaussian-displacement-body}.

For general Lipschitz \(a,b\), we choose a standard mollifier family
\((\zeta_\varepsilon)_{\varepsilon>0}\), and set \(a_n=a*\zeta_{1/n}\) and
\(b_n=b*\zeta_{1/n}\). By \Cref{lem:width-Bregman}, \(\cK_{a_n}\subseteq\cK_a\) and \(\cK_{b_n}\subseteq\cK_b\). The Gaussian specialisation of the weighted-density and stability argument from \Cref{sec:general-proof}, with \Cref{cor:gaussian-bilip}, yields
\[
T_{a_n,b_n}\longrightarrow T_{a,b}
\qquad\text{locally uniformly}.
\]
We pass to the limit in the closed inclusion
\[
T_{a_n,b_n}(x)-x\in\cK_a-\cK_b
\]
and obtain \eqref{eq:gaussian-displacement-body}. The norm estimate follows
from \(\cK_f\subseteq\overline B_{\Lip(f)}(0)\) for \(f\in\{a,b\}\). Finally, let \(a(x)=\ip{p}{x}+c_a\) and \(b(y)=\ip{q}{y}+c_b\). By completing the square, we obtain \(T_{a,b}(x)=x+p-q\).
\end{proof}

\section{Compact Gaussian mixtures}\label{sec:mixtures}

\begin{theorem}[Compact Gaussian mixtures]\label{thm:mixtures}
Let \(\eta\) be a compactly supported probability measure on \(\R^d\), and set
\[
K:=\supp\eta,
\qquad
R:=\diam K,
\qquad
m_\eta:=\int z\,\eta(\dd z),
\qquad
\rho_\eta:=\sup_{z\in K}\norm{z-m_\eta}.
\]
We also set
\begin{equation}\label{eq:def-Meta}
\cM(\eta)
:=\min\left\{
\cC(0,\rho_\eta),
1+(1+4R^2)\e^{4R^2}
\right\}.
\end{equation}
Let \(T=\nabla\Phi\) be the Brenier map from \(\gamma_d\) to \(\gamma_d*\eta\). It follows that \(T\) is a smooth diffeomorphism and
\[
T(x)-x\in\conv K
\qquad(x\in\R^d).
\]
When \(R=0\), the map \(T\) is a translation. If \(R>0\), we obtain
\begin{equation}\label{eq:mixture-Hessian-intro}
\Id\preceq D^2\Phi(x)\preceq\cM(\eta)\Id
\qquad(x\in\R^d).
\end{equation}
The anisotropic estimate in terms of the directional widths of \(K\) is given in \Cref{cor:mixture-spectral-bound}.
\end{theorem}

We fix a compactly supported Borel probability measure \(\eta\) on \(\R^d\). Let
\(T=\nabla\Phi\) denote the Brenier map from \(\gamma_d\) to \(\nu:=\gamma_d*\eta\). The anisotropic support-width estimate follows from \Cref{cor:mixture-spectral-bound}, and the \(\cC(0,\rho_\eta)\) bound in \eqref{eq:def-Meta} follows directly from the Gaussian perturbation result.

\begin{lemma}[Centring]\label{lem:mixture-centering}
Let
\[
\bar\eta:=(z\mapsto z-m_\eta)_\#\eta,
\qquad
\bar K:=K-m_\eta.
\]
It follows that \(\bar\eta\) is centred, and
\[
\bar K\subseteq\overline B_{\rho_\eta}(0),
\qquad
\rho_\eta\le R,
\qquad
\diam\bar K=R.
\]
We set \(\bar T:=T-m_\eta\), so that \(\bar T\) is the Brenier map from
\(\gamma_d\) to \(\gamma_d*\bar\eta\), and \(D\bar T=DT\).
\end{lemma}

\begin{proof}
For \(z\in K\),
\[
\norm{z-m_\eta}
=\norm{\int_K(z-z')\eta(\dd z')}
\le\int_K\norm{z-z'}\eta(\dd z')
\le R,
\]
which proves \(\rho_\eta\le R\). The measure \(\gamma_d*\bar\eta\) is the pushforward of \(\gamma_d*\eta\) under the translation \(y\mapsto y-m_\eta\). Also, the map \(T-m_\eta\) is the gradient of the convex function \(\Phi(x)-\ip{m_\eta}{x}\), and it is the Brenier map to the translated target.
\end{proof}

By \Cref{lem:mixture-centering}, it suffices to establish all derivative estimates under the normalisation
\begin{equation}\label{eq:centered-mixture}
\int z\,\eta(\dd z)=0,
\qquad
K\subseteq\overline B_{\rho_\eta}(0).
\end{equation}

Let \(\varphi_d\) denote the standard Gaussian density. The density of \(\nu=\gamma_d*\eta\) is given by
\begin{align}
q_\eta(y)
&=\int_K\varphi_d(y-z)\eta(\dd z)\notag\\
&=\varphi_d(y)
\int_K\exp\!\left(\ip{y}{z}-\frac12\norm{z}^2\right)\eta(\dd z).
\label{eq:mixture-density}
\end{align}
We define
\begin{equation}\label{eq:psi-mixture}
\psi(y)
:=\log\int_K
\exp\!\left(\ip{y}{z}-\frac12\norm{z}^2\right)\eta(\dd z).
\end{equation}
We then have
\begin{equation}\label{eq:mixture-relative-density}
q_\eta(y)=\varphi_d(y)\e^{\psi(y)}.
\end{equation}
Up to an additive constant, the Lebesgue potential of \(\nu\) is
\begin{equation}\label{eq:Ueta-mixture}
U_\eta(y):=\frac12\norm{y}^2-\psi(y).
\end{equation}
For \(y\in\R^d\), we introduce the tilted mixing law
\begin{equation}\label{eq:tilted-mixture}
\eta_y(\dd z)
:=\frac{\exp(\ip{y}{z}-\norm{z}^2/2)}
{\int_K\exp(\ip{y}{z'}-\norm{z'}^2/2)\eta(\dd z')}
\eta(\dd z).
\end{equation}
Let \(Z\) be the identity map on \(K\), which one may understand as a random vector under \(\eta_y\). Differentiation under the integral gives
\begin{equation}\label{eq:psi-derivatives-mixture}
\nabla\psi(y)=\mathbb E_{\eta_y}Z,
\qquad
D^2\psi(y)=\Cov_{\eta_y}(Z).
\end{equation}
These are the Gaussian conditional-mean and conditional-covariance
identities. For this interpretation, see \cite{Miyasawa} and \cite[eqs. (3) and (17)]{DytsoPoorShamai}. For the unit-noise divergence and conditional-variance trace identity, consult \cite[Property 3 and eqs. (13)--(14)]{HatsellNolte}. The identities
imply
\begin{equation}\label{eq:psi-gradient-mixture}
\nabla\psi(y)\in\conv K,
\qquad
\norm{\nabla\psi(y)}\le\rho_\eta,
\end{equation}
and
\begin{equation}\label{eq:Ueta-semiconcavity-mixture}
D^2U_\eta(y)=\Id-\Cov_{\eta_y}(Z)\preceq\Id.
\end{equation}
For a nonempty compact set \(K'\subset\R^d\) and \(v\in\R^d\), we write
\begin{equation}\label{eq:directional-width-mixture}
w_{K'}(v):=\sup_{z,z'\in K'}\ip{z-z'}{v}.
\end{equation}
Set \(b=-\psi\) in the Gaussian perturbation representation of the target. We then have
\begin{equation}\label{eq:mixture-target-perturbation-data}
\Lip(b)\le\rho_\eta,\qquad
\cK_b\subseteq-\conv K,\qquad
\omega_b(v)\le w_K(v)\quad(v\in\R^d).
\end{equation}
Notice that the source and target densities are smooth and strictly positive on \(\R^d\).
By \Cref{prop:smooth-diffeo-general}, the map \(T\) is a smooth diffeomorphism.

\begin{proposition}[Inverse contraction]\label{prop:inverse-contraction-mixture}
For every \(x\in\R^d\), it holds that
\begin{equation}\label{eq:A-lower-mixture}
D^2\Phi(x)\succeq\Id.
\end{equation}
In particular,
\begin{equation}\label{eq:h-S-B-mixture}
h(x):=\Phi(x)-\frac12\norm{x}^2,
\qquad
S(x):=\nabla h(x)=T(x)-x,
\qquad
B(x):=D^2h(x)=D^2\Phi(x)-\Id
\end{equation}
define, respectively, a smooth convex function, its gradient, and a positive-semidefinite matrix field.
\end{proposition}

\begin{proof}
Let \(\widehat T:=T^{-1}=\nabla\Phi^*\) denote the inverse Brenier map from
\(\nu\) to \(\gamma_d\). Its source potential is \(U_\eta\), and its target potential is \(\norm{\cdot}^2/2\). By \eqref{eq:Ueta-semiconcavity-mixture}, the source Hessian is bounded above by \(\Id\), while the target Hessian equals \(\Id\). The general Hessian comparison (\textit{cf}. \Cref{sec:introduction}) yields \(\Lip(\widehat T)\le1\). Since
\(D\widehat T\) is symmetric and positive definite, it follows that
\[
0\prec D\widehat T(y)\preceq\Id.
\]
Moreover, at \(y=T(x)\), the inverse-function identity gives \(D\widehat T(T(x))=(D^2\Phi(x))^{-1}\). This establishes \eqref{eq:A-lower-mixture}.
\end{proof}

\begin{corollary}[Spectral support-width bound for mixtures]
\label{cor:mixture-spectral-bound}
Under the centred normalisation, we fix \(\Sigma\in\mathrm{Sym}_d^{++}\) and
\(0<\beta<1\) such that
\begin{equation}\label{eq:mixture-spectral-admissibility}
\Sigma-\beta^2\Sigma^{-1}\succ0.
\end{equation}
Define
\begin{equation}\label{eq:mixture-spectral-action}
\cA_K(\Sigma,\beta)
:=\int_0^\infty\sup_{u\in\mathbb S^{d-1}}
\left(
w_K(\Sigma^{1/2}u)
-\frac t2\ip{(\Sigma-\beta^2\Sigma^{-1})u}{u}
\right)_+\dd t.
\end{equation}
It follows that
\begin{equation}\label{eq:mixture-spectral-upper}
D^2\Phi
\preceq
\frac{\e^{\cA_K(\Sigma,\beta)}}{\beta}\Sigma.
\end{equation}
\end{corollary}

\begin{proof}
We apply \Cref{thm:spectral-caffarelli} with \(Q=P=\Id\), \(a=0\), and
\(H=\beta^{-1}\Sigma\). It follows that \(G=\mathsf R=\Id\), and the condition
\(H\mathsf RH\succ\mathsf R\) is precisely \eqref{eq:mixture-spectral-admissibility}. Moreover,
\[
\mathsf D_H=\beta^{-1}(\Sigma-\beta^2\Sigma^{-1}),
\qquad
\omega_b(H^{1/2}u)
\le\beta^{-1/2}w_K(\Sigma^{1/2}u),
\]
by \eqref{eq:mixture-target-perturbation-data}. We make the change of
variables \(s=t/\sqrt\beta\) in the reduced spectral action, and obtain
\(\cA_{Q,P}^{\mathrm{red}}(H,0,b)\le\cA_K(\Sigma,\beta)\). The estimate now follows.
\end{proof}

\begin{proposition}[Compact displacement range]
\label{prop:compact-displacement}
Under \eqref{eq:centered-mixture}, we have
\begin{equation}\label{eq:S-range-mixture}
S(\R^d)\subseteq\conv K.
\end{equation}
Without the centring assumption, this gives
\begin{equation}\label{eq:T-minus-Id-range}
T(x)-x\in\conv K
\qquad(x\in\R^d).
\end{equation}
\end{proposition}

\begin{proof}
For the centred assertion, we apply \Cref{thm:gaussian-displacement} with
\(a=0\) and use \eqref{eq:mixture-target-perturbation-data} to obtain
\[
T(x)-x\in-\cK_b\subseteq\conv K.
\]
For the original mixing law, \Cref{lem:mixture-centering} gives \(\bar T=T-m_\eta\) and \(\bar K=K-m_\eta\). We conclude that
\[
T(x)-x\in\conv(K-m_\eta)+m_\eta=\conv K.
\]
\end{proof}

\subsection{Penalised Bregman estimates}
\label{subsec:mixture-penalised-bregman-estimates}

\begin{lemma}[Mixture Bregman estimate]\label{lem:mixture-Bregman}
Recall the Bregman-remainder notation \(D_f\) from
\eqref{eq:Bregman-remainder-notation}. For every \(y,v\in\R^d\), it holds that
\begin{equation}\label{eq:mixture-Bregman-lower}
D_{U_\eta}(y+v,y)
\ge \frac12\norm{v}^2-R\norm{v}.
\end{equation}
\end{lemma}

\begin{proof}
Equation \eqref{eq:Ueta-mixture} gives
\begin{equation}\label{eq:Ueta-versus-psi-Bregman}
D_{U_\eta}(y+v,y)=\frac12\norm{v}^2-D_\psi(y+v,y).
\end{equation}
The tilted law from \eqref{eq:tilted-mixture} gives
\begin{align}
D_\psi(y+v,y)
&=\log\mathbb E_{\eta_y}\e^{\ip{v}{Z}}
-\mathbb E_{\eta_y}\ip{v}{Z}\notag\\
&=\log\mathbb E_{\eta_y}
\exp\!\left(\ip{v}{Z-\mathbb E_{\eta_y}Z}\right).
\label{eq:psi-centered-mgf}
\end{align}
The scalar random variable \(\ip{v}{Z}\) has oscillation at most
\[
\sup_{z,z'\in K}\ip{v}{z-z'}\le R\norm{v}.
\]
It follows that
\[
\ip{v}{Z-\mathbb E_{\eta_y}Z}
\le R\norm{v}
\qquad \eta_y\text{-almost surely},
\]
and \eqref{eq:psi-centered-mgf} now gives \(D_\psi(y+v,y)\le R\norm{v}\). We insert this estimate into \eqref{eq:Ueta-versus-psi-Bregman}, and arrive at \eqref{eq:mixture-Bregman-lower}.
\end{proof}

For the remainder of this subsection, we will assume that \(R>0\). Let us set
\begin{equation}\label{eq:mixture-lambda-Lambda}
\lambda_R:=4R,
\qquad
\Lambda_R:=(1+\lambda_RR)\e^{\lambda_RR}
=(1+4R^2)\e^{4R^2}.
\end{equation}
We define
\begin{equation}\label{eq:mixture-rho-curve}
\mathscr Q(q):=\frac{q\e^{\lambda_Rq}}{\Lambda_R}
\qquad(0\le q\le R).
\end{equation}
The function \(\mathscr Q\) is strictly increasing, and
\begin{equation}\label{eq:mixture-r-star}
r_*:=\mathscr Q(R)=\frac{R}{1+\lambda_RR}.
\end{equation}
Let \(p_R:[0,\infty)\to[0,\infty)\) be defined by
\begin{equation}\label{eq:mixture-p-inverse}
p_R(r):=\mathscr Q^{-1}(r)
\qquad(0\le r\le r_*),
\end{equation}
and extend it to \(r\ge r_*\) by
\begin{equation}\label{eq:mixture-p-extension}
p_R(r):=R+(r-r_*).
\end{equation}
Finally, put
\begin{equation}\label{eq:mixture-theta-Theta}
\theta_R(r):=\int_0^r p_R(s)\dd s,
\qquad
\Theta_R(m):=\theta_R(\norm{m}).
\end{equation}

\begin{lemma}[Properties of the diameter penalty]
\label{lem:mixture-penalty-properties}
The function \(p_R\) is \(C^1\) and strictly increasing, with
\begin{equation}\label{eq:mixture-pprime-zero}
p_R(0)=0,
\qquad
p_R'(0)=\Lambda_R.
\end{equation}
The radial function \(\Theta_R\) belongs to \(C^2(\R^d)\), is strictly convex, and satisfies
\begin{equation}\label{eq:mixture-penalty-coercive}
\theta_R(r)-Rr\longrightarrow+\infty
\qquad(r\to\infty).
\end{equation}
Further assume that \(r>0\) and \(q:=p_R(r)\le R\). With
\begin{equation}\label{eq:mixture-tau-def}
\tau:=\frac qr,
\end{equation}
we have
\begin{equation}\label{eq:mixture-penalty-identities}
\tau=\Lambda_R\e^{-\lambda_Rq},
\qquad
p_R'(r)=\frac{\tau}{1+\lambda_Rq},
\qquad
\tau\ge1+\lambda_RR\ge1+\lambda_Rq>1.
\end{equation}
\end{lemma}

\begin{proof}
We differentiate \eqref{eq:mixture-rho-curve} and obtain
\[
\mathscr Q'(q)
=\frac{\e^{\lambda_Rq}(1+\lambda_Rq)}{\Lambda_R}>0.
\]
At \(q=R\), \(\mathscr Q'(R)=1\) holds. The inverse branch \eqref{eq:mixture-p-inverse} joins the affine extension \eqref{eq:mixture-p-extension} with matching first derivative
\cite[Section 4.3 and Proposition 4.8]{Gwozdz}, so \(p_R\in C^1\). At the origin,
\[
p_R'(0)=\frac1{\mathscr Q'(0)}=\Lambda_R.
\]
Moreover, \(p_R(0)=0\) and \(p_R'>0\). We apply \Cref{lem:radial-penalty-calculus} with \(S=\Id\), so it follows that \(\Theta_R\in C^2(\R^d)\) and is strictly convex. For \(m\ne0\), we write \(r:=\norm m\) and \(e:=m/r\). We infer
\begin{equation}\label{eq:mixture-radial-Hessian}
D^2\Theta_R(m)
=p_R'(r)e\otimes e
+\frac{p_R(r)}r(\Id-e\otimes e),
\end{equation}
and the affine extension of \(p_R\) gives quadratic growth of \(\theta_R\). This
proves \eqref{eq:mixture-penalty-coercive}.

Now suppose that \(q=p_R(r)\le R\). It then holds that \(r=\mathscr Q(q)\), so
\[
\tau=\frac qr=\Lambda_R\e^{-\lambda_Rq}.
\]
Differentiating the inverse gives
\[
p_R'(r)=\frac1{\mathscr Q'(q)}
=\frac{\Lambda_R\e^{-\lambda_Rq}}{1+\lambda_Rq}
=\frac{\tau}{1+\lambda_Rq}.
\]
Since \(q\le R\),
\[
\tau\ge\Lambda_R\e^{-\lambda_RR}
=1+\lambda_RR\ge1+\lambda_Rq.
\]
This completes the proof.
\end{proof}

Instead of the compact target range we previously used in \cite[Section 8.1]{Gwozdz}, we now consider the compact displacement range from \Cref{prop:compact-displacement}, together with the above penalty. Under the centred normalisation \eqref{eq:centered-mixture}, we define
\begin{equation}\label{eq:mixture-F-def}
F_m(x)
:=h(x+m)-h(x)-\ip{m}{S(x)}.
\end{equation}
Since \(h\) is convex,
\begin{equation}\label{eq:mixture-F-lower}
F_m(x)\ge0.
\end{equation}
Moreover, \Cref{prop:compact-displacement} and
\(\diam(\conv K)=\diam K=R\) give
\begin{align}
F_m(x)
&=\int_0^1\ip{S(x+sm)-S(x)}{m}\dd s\notag\\
&\le R\norm{m}.
\label{eq:mixture-F-upper}
\end{align}

For fixed \(x,m\), we set
\begin{equation}\label{eq:mixture-corrector-matrices}
A:=D^2\Phi(x),
\qquad
C:=D^2\Phi(x+m),
\qquad
\mathcal R:=A^{-1/2}CA^{-1/2}.
\end{equation}
For \(f\in C^2(\R^d)\), let us also introduce the linearised operator
\begin{equation}\label{eq:mixture-linearized-operator}
(\mathcal L f)(z)
:=\tr\!\left((D^2\Phi(z))^{-1}D^2f(z)\right)
-\ip{\nabla U_\eta(T(z))}{\nabla f(z)}
\qquad(z\in\R^d).
\end{equation}

Since \(F_m=\mathfrak B_m-\norm{m}^2/2\), the translated identity
\eqref{eq:LG-exact-general}, which we apply with the Gaussian source potential
\(\norm{x}^2/2\), gives, for every \(x,m\in\R^d\),
\begin{equation}\label{eq:mixture-translated-identity}
\mathcal L F_m
=\cH(\mathcal R)-\frac12\norm{m}^2
+D_{U_\eta}(T(x+m),T(x)).
\end{equation}

\begin{proposition}[Global comparison]
\label{prop:mixture-F-Theta}
For all \(x,m\in\R^d\), it holds that
\begin{equation}\label{eq:mixture-F-versus-Theta}
F_m(x)\le\Theta_R(m).
\end{equation}
\end{proposition}

\begin{proof}
Suppose, for the sake of contradiction, that there exist \(x_0,m_0\in\R^d\) and \(\delta>0\) such that
\begin{equation}\label{eq:mixture-positive-gap}
F_{m_0}(x_0)-\Theta_R(m_0)=\delta.
\end{equation}
By \Cref{prop:inverse-contraction-mixture}, the function \(\Phi\) is strongly
convex. In particular, it is coercive and has a unique minimum. We normalise
\(\Phi\) so that \(\min\Phi=0\). For \(\varepsilon>0\), let us define
\begin{equation}\label{eq:mixture-J-def}
J_\varepsilon(x,m)
:=F_m(x)-\Theta_R(m)-\varepsilon\Phi(x).
\end{equation}
The bounds \eqref{eq:mixture-F-upper} and \eqref{eq:mixture-penalty-coercive}, with the coercivity of \(\Phi\), show that \(J_\varepsilon\) attains a global maximum at a pair
\((x_\varepsilon,m_\varepsilon)\). For all sufficiently small \(\varepsilon\), we have
\begin{equation}\label{eq:mixture-J-positive}
J_\varepsilon(x_\varepsilon,m_\varepsilon)\ge\frac\delta2>0.
\end{equation}
In particular, \(m_\varepsilon\ne0\). We now write
\begin{equation}\label{eq:mixture-r-e-q}
r_\varepsilon:=\norm{m_\varepsilon},
\qquad
e_\varepsilon:=\frac{m_\varepsilon}{r_\varepsilon},
\qquad
q_\varepsilon:=p_R(r_\varepsilon),
\end{equation}
and omit the subscript \(\varepsilon\) from
\(x_\varepsilon,m_\varepsilon,r_\varepsilon,e_\varepsilon,q_\varepsilon\) in
the the calculations.

From \eqref{eq:mixture-F-upper} and \eqref{eq:mixture-J-positive}, we have
\begin{equation}\label{eq:mixture-r-compact}
\frac\delta2
\le J_\varepsilon(x,m)
\le Rr-\theta_R(r).
\end{equation}
Notice that the radii \(r=r_\varepsilon\) remain in a compact subset of
\((0,\infty)\), and the same inequalities show that \(\varepsilon\Phi(x_\varepsilon)\) is uniformly bounded. Since \(\Phi\) is \(1\)-strongly convex by \Cref{prop:inverse-contraction-mixture}, we arrive at
\begin{equation}\label{eq:mixture-epsilon-x-zero}
\varepsilon\norm{x_\varepsilon}\longrightarrow0.
\end{equation}
Moreover, \(T(x)=x+S(x)\) and \(S(x)\in\conv K\subseteq\overline B_{\rho_\eta}(0)\). It follows that
\begin{equation}\label{eq:mixture-epsilon-T-zero}
\varepsilon\norm{T(x_\varepsilon)}\longrightarrow0.
\end{equation}

We apply \Cref{lem:quadratic-baseline-calculus} with \(\Gamma=\Id\),
\(\Theta=\Theta_R\), and \(F^\Id_m=F_m\). The stationarity condition in \(m\)
gives
\begin{equation}\label{eq:mixture-m-stationarity}
S(x+m)-S(x)=\nabla\Theta_R(m)=p_R(r)e=qe.
\end{equation}
Both values of \(S\) belong to \(\conv K\), so
\begin{equation}\label{eq:mixture-q-bound}
0<q\le R.
\end{equation}
We combine \eqref{eq:mixture-r-e-q} with \eqref{eq:mixture-m-stationarity} to obtain
\begin{equation}\label{eq:mixture-T-increment}
T(x+m)-T(x)=m+S(x+m)-S(x)=(r+q)e.
\end{equation}

The stationarity condition in \(x\) gives \(D^2\Phi(x)m=T(x+m)-T(x)-\varepsilon T(x)\). We then obtain
\begin{equation}\label{eq:mixture-Ae-relation}
D^2\Phi(x)e
=\left(1+\frac qr\right)e-\frac\varepsilon rT(x),
\end{equation}
so
\begin{equation}\label{eq:mixture-Aee-relation}
\ip{D^2\Phi(x)e}{e}
=1+\frac qr-\frac\varepsilon r\ip{T(x)}{e}.
\end{equation}

At the maximum, set
\begin{equation}\label{eq:mixture-A-C-M}
A:=D^2\Phi(x),
\qquad
C:=D^2\Phi(x+m),
\qquad
M:=\Id+D^2\Theta_R(m).
\end{equation}
The baseline calculus gives \(C\preceq M\), together with
\begin{equation}\label{eq:mixture-compatibility}
(C-A)z=0
\qquad(z\in\Ker(M-C))
\end{equation}
Let us set
\begin{equation}\label{eq:mixture-Schur-term}
\cS_\varepsilon^{\mathrm{mix}}
:=\cS(A,C,M)
=\tr\!\left[
A^{-1}(C-A)(M-C)^\dagger(C-A)
\right].
\end{equation}
In this case, \eqref{eq:baseline-L-Schur} becomes
\begin{equation}\label{eq:mixture-L-plus-S}
\mathcal LF_m+\cS_\varepsilon^{\mathrm{mix}}
\le\varepsilon\left(d-\ip{\nabla U_\eta(T(x))}{T(x)}\right).
\end{equation}
Finally, the identities \(\nabla U_\eta(T(x))=T(x)-\nabla\psi(T(x))\) and
\(\norm{\nabla\psi(T(x))}\le\rho_\eta\) give
\begin{equation}\label{eq:mixture-LPhi-bound}
d-\ip{\nabla U_\eta(T(x))}{T(x)}
\le d+\frac{\rho_\eta^2}{4}.
\end{equation}

We now combine \eqref{eq:mixture-translated-identity}, \eqref{eq:mixture-L-plus-S}, and \eqref{eq:mixture-LPhi-bound}, which yields
\begin{equation}\label{eq:mixture-H-S-before-Bregman}
\cH(\mathcal R)+\cS_\varepsilon^{\mathrm{mix}}
\le\frac{r^2}{2}
-D_{U_\eta}(T(x+m),T(x))
+\varepsilon\left(d+\frac{\rho_\eta^2}{4}\right).
\end{equation}
By \Cref{lem:mixture-Bregman} and \eqref{eq:mixture-T-increment},
\[
D_{U_\eta}(T(x+m),T(x))
\ge\frac{(r+q)^2}{2}-R(r+q).
\]
We define
\begin{equation}\label{eq:mixture-E-def}
E_R(r,q)
:=\frac{r^2}{2}-\frac{(r+q)^2}{2}+R(r+q)
=r(R-q)+q\left(R-\frac q2\right).
\end{equation}
By the above estimates, we have
\begin{equation}\label{eq:mixture-H-S-upper-E}
\cH(\mathcal R)+\cS_\varepsilon^{\mathrm{mix}}
\le E_R(r,q)
+\varepsilon\left(d+\frac{\rho_\eta^2}{4}\right).
\end{equation}

We now apply \Cref{lem:matrix-coercivity} to \(A,C,M\) and \(e\). The compatibility condition is exactly \eqref{eq:mixture-compatibility}, and equation \eqref{eq:mixture-radial-Hessian} gives
\begin{equation}\label{eq:mixture-Mee}
\ip{Me}{e}=1+p_R'(r).
\end{equation}
The lemma leads to
\begin{equation}\label{eq:mixture-coercivity-at-max}
y_\varepsilon:=\frac{\ip{Ae}{e}}{1+p_R'(r)},
\qquad
\cH(\mathcal R)+\cS_\varepsilon^{\mathrm{mix}}
\ge\Xi_+(y_\varepsilon).
\end{equation}
We combine \eqref{eq:mixture-H-S-upper-E} with \eqref{eq:mixture-coercivity-at-max}, and arrive at
\begin{equation}\label{eq:mixture-scalar-epsilon}
\Xi_+(y_\varepsilon)
\le E_R(r,q)
+\varepsilon\left(d+\frac{\rho_\eta^2}{4}\right).
\end{equation}

Let \(\varepsilon\downarrow0\) along an arbitrary sequence. By \eqref{eq:mixture-r-compact}, after passing to a subsequence, we may assume
\[
r_\varepsilon\longrightarrow r>0,
\qquad
q_\varepsilon=p_R(r_\varepsilon)\longrightarrow q=p_R(r)\in(0,R].
\]
Equations \eqref{eq:mixture-Aee-relation} and \eqref{eq:mixture-epsilon-T-zero} give
\begin{equation}\label{eq:mixture-y-limit}
y_\varepsilon\longrightarrow
Y:=\frac{1+q/r}{1+p_R'(r)}.
\end{equation}
Let us put \(\tau:=q/r\). By \Cref{lem:mixture-penalty-properties}, \(p_R'(r)=\tau/(1+\lambda_Rq)\), so
\begin{equation}\label{eq:mixture-Y-formula}
Y=\frac{1+\tau}{1+\tau/(1+\lambda_Rq)},
\qquad
Y-1=\frac{\lambda_Rq\tau}{1+\lambda_Rq+\tau}.
\end{equation}
Since \(\tau\ge1+\lambda_Rq\),
\begin{equation}\label{eq:mixture-Y-lower}
Y-1\ge\frac{\lambda_Rq}{2}=2Rq.
\end{equation}
In particular, \(Y>1\) is true. Moreover,
\[
\Xi_+'(s)=\left(1+\frac1s\right)^2\ge1
\qquad(s>1),
\qquad
\Xi_+(1)=0,
\]
and
\begin{equation}\label{eq:mixture-Xi-lower}
\Xi_+(Y)\ge Y-1\ge2Rq.
\end{equation}
On the other hand, \(r=q/\tau\) and \(q\le R\), so
\begin{align}
E_R(r,q)
&\le Rr+Rq\notag\\
&=Rq\left(1+\frac1\tau\right)<2Rq,
\label{eq:mixture-E-upper}
\end{align}
because \(\tau>1\). Passing to the limit in \eqref{eq:mixture-scalar-epsilon} gives
\[
\Xi_+(Y)\le E_R(r,q),
\]
which contradicts \eqref{eq:mixture-Xi-lower} and \eqref{eq:mixture-E-upper}. This proves the proposition.
\end{proof}

We now combine the above estimates to finish the proof of the mixture
theorem.

\begin{proof}[Proof of \Cref{thm:mixtures}]
The displacement inclusion follows from \Cref{prop:compact-displacement},
and the lower Hessian estimate is \Cref{prop:inverse-contraction-mixture}.
Let us first consider \(R=0\). In this case, \(K\) is a singleton. The inclusion
\(T(x)-x\in\conv K\) then shows that \(T\) is the corresponding translation.

Assume now that \(R>0\). Using the centred normalisation, fix \(x\in\R^d\) and \(e\in\mathbb S^{d-1}\). For \(t>0\), \Cref{prop:mixture-F-Theta} gives
\begin{equation}\label{eq:mixture-Taylor-comparison}
F_{te}(x)\le\Theta_R(te)=\theta_R(t).
\end{equation}
Since \(h\in C^2\),
\[
F_{te}(x)
=\frac{t^2}{2}\ip{B(x)e}{e}+o(t^2).
\]
By \eqref{eq:mixture-pprime-zero},
\[
\theta_R(t)=\frac{\Lambda_Rt^2}{2}+o(t^2).
\]
We divide \eqref{eq:mixture-Taylor-comparison} by \(t^2/2\), and then let
\(t\downarrow0\). This gives
\[
\ip{B(x)e}{e}\le\Lambda_R.
\]
Since \(e\) is arbitrary and \(B\succeq0\),
\begin{equation}\label{eq:mixture-special-Hessian}
\Id\preceq D^2\Phi(x)
\preceq\bigl[1+(1+4R^2)\e^{4R^2}\bigr]\Id.
\end{equation}
The other upper bound in \eqref{eq:def-Meta} follows by combining \Cref{thm:gaussian-source} with \eqref{eq:mixture-target-perturbation-data}. Using
\eqref{eq:mixture-special-Hessian}, this gives \eqref{eq:mixture-Hessian-intro} under the centred normalisation. Note that translation of the target leaves \(D^2\Phi\) unchanged, so the same estimate holds without centring. Finally,
\[
w_{K-m_\eta}=w_K,
\]
and the spectral support-width estimate from \Cref{cor:mixture-spectral-bound} also holds for the original mixing law.
\end{proof}

\appendix

\section{Exact spectral Schur envelope}
\label{app:exact-schur-envelopes}

In \Cref{app:exact-schur-envelopes}, we determine the trace-level envelope used in the proof of \Cref{lem:matrix-coercivity}.

\begin{proposition}[Exact spectral Schur envelope]
\label{prop:spectral-schur-envelope}
Let \(A,P\in\mathrm{Sym}_d^{++}\). For \(0\prec C\prec P\), we define
\begin{align}
\mathcal F_{A,P}(C)
:={}&\tr(A^{-1}C)-d-\log\det(A^{-1}C)
\notag\\
&+\tr\!\left[
A^{-1}(C-A)(P-C)^{-1}(C-A)
\right].
\label{eq:matrix-Schur-functional}
\end{align}
It follows that
\begin{equation}\label{eq:exact-spectral-envelope}
\inf_{0\prec C\prec P}\mathcal F_{A,P}(C)
=\tr\,\Xi_+\!\left(P^{-1/2}AP^{-1/2}\right),
\end{equation}
where \(\Xi_+\) is applied by functional calculus.
\end{proposition}

\begin{proof}
We first apply congruence by \(P^{-1/2}\) and reduce the problem to \(P=\Id\). Let us also set
\[
K:=A+A^{-1}-2\Id\succeq0.
\]
Expansion and cyclicity of the trace give
\begin{equation}\label{eq:F-expanded}
\mathcal F_{A,\Id}(C)
=\tr\!\left(K(\Id-C)^{-1}\right)
-\log\det C
+d-\tr A^{-1}+\log\det A.
\end{equation}
Let us first assume that \(1\) is not an eigenvalue of \(A\), so \(K\succ0\). The \(C\)-dependent part of \eqref{eq:F-expanded} is strictly convex on \(0\prec C\prec\Id\). To verify this, set \(X:=\Id-C\). For a symmetric variation \(E\), the second variation is
\[
2\tr\!\left(KX^{-1}EX^{-1}EX^{-1}\right)
+\tr\!\left(C^{-1}EC^{-1}E\right),
\]
which is strictly positive whenever \(E\ne0\). Indeed, if
\[
F:=X^{-1/2}EX^{-1/2},
\qquad
B:=X^{-1/2}KX^{-1/2}\succ0,
\]
then the first trace equals \(2\tr(BF^2)\ge0\), while the second trace is strictly positive.

Observe that the functional tends to \(+\infty\) at both boundary components, and its critical point is the unique minimiser. The Euler equation is
\begin{equation}\label{eq:spectral-Euler}
(\Id-C)^{-1}K(\Id-C)^{-1}=C^{-1},
\end{equation}
and after multiplication on the left and right by \(\Id-C\), this becomes
\begin{equation}\label{eq:K-C-equation}
K=C+C^{-1}-2\Id.
\end{equation}
Equation \eqref{eq:K-C-equation} implies that \(C\) commutes with \(K\). It follows that \(C\) and \(K\) are simultaneously orthogonally diagonalisable. On each common eigenspace, the eigenvalues \(c\in(0,1)\) and \(k>0\) satisfy
\[
c+c^{-1}-2=k.
\]
For every \(k>0\), this scalar equation has a unique solution \(c=c(k)\in(0,1)\). In other terms,
\[
C=c(K)
\]
by functional calculus. Since
\[
K=A+A^{-1}-2\Id
\]
and
\[
c\!\left(t+t^{-1}-2\right)=\min\{t,t^{-1}\},
\qquad t>0,\quad t\ne1,
\]
the two identities give
\[
\chi(t):=\min\{t,t^{-1}\},
\qquad
C=\chi(A).
\]
This also covers the eigenspaces on which reciprocal eigenvalues of \(A\) give the same eigenvalue of \(K\), because the corresponding value of \(\chi\) is the same. We conclude that the scalar contribution of an eigenvalue \(t\) of \(A\) is zero for \(t<1\), and equals \(t+2\log t-t^{-1}\) for \(t>1\).

It remains to remove the assumption \(K\succ0\). For \(\delta>0\), let us set
\[
K_\delta:=K+\delta\Id
\]
and define
\[
\mathcal F_\delta(C)
:=
\tr\!\left(K_\delta(\Id-C)^{-1}\right)
-\log\det C
+d-\tr A^{-1}+\log\det A.
\]
It follows that \(\mathcal F_\delta(C)\downarrow\mathcal F_{A,\Id}(C)\) pointwise as
\(\delta\downarrow0\), and
\[
\inf_{0\prec C\prec\Id}\mathcal F_\delta(C)
\longrightarrow
\inf_{0\prec C\prec\Id}\mathcal F_{A,\Id}(C).
\]
The left-hand side is bounded below by the limiting infimum. For the reverse
limsup inequality, it suffices to evaluate at a fixed approximate minimiser of \(\mathcal F_{A,\Id}\).

For each \(\delta>0\), the previous computation also applies to \(K_\delta=K+\delta\Id\), and the unique minimiser is \(c(K_\delta)\). Since \(K\) is a continuous function of \(A\), spectral calculus gives \(c(K_\delta)\to\chi(A)\), where \(\chi(t)=\min\{t,t^{-1}\}\) and \(\chi(1)=1\) on the \(t=1\) eigenspace. Note that the scalar summand converges to zero
for \(t\le1\), and to \(t+2\log t-t^{-1}\) for \(t>1\). Finally, we sum over the spectrum and use the convergence of the infima. This proves \eqref{eq:exact-spectral-envelope}.
\end{proof}

\bibliographystyle{plain}
\bibliography{log}

\end{document}